\UseRawInputEncoding 
\documentclass[12pt]{amsart}       
\usepackage{txfonts}
\usepackage{algorithm}
\usepackage{algorithmic}
\usepackage{amssymb}
\usepackage{eucal}
\usepackage{graphicx}
\usepackage{amsmath}
\usepackage{amscd}
\usepackage[all]{xy}           
\usepackage{tikz}
\usepackage{tikz-qtree}
\usepackage{amsfonts,latexsym}
\usepackage{xspace}
\usepackage{epsfig}
\usepackage{float}
\usepackage{color}
\usepackage{fancybox}
\usepackage{colordvi}
\usepackage{multicol}
\usepackage{colordvi}
\usepackage{wasysym}
\usepackage{makecell} 
\usepackage[numbers,sort&compress]{natbib}
\usepackage[active]{srcltx} 
\usepackage{CJK}
\ifpdf
  \usepackage[colorlinks,final,backref=page,hyperindex]{hyperref}
\else
  \usepackage[colorlinks,final,backref=page,hyperindex,hypertex]{hyperref}
\fi

\newcommand{\nc}{\newcommand}
\newcommand{\delete}[1]{}
\nc{\bfk}{{\bf k}}

\nc{\mlabel}[1]{\label{#1}}  
\nc{\mcite}[1]{\cite{#1}}  
\nc{\mref}[1]{\ref{#1}}  
\nc{\mbibitem}[1]{\bibitem{#1}} 

\delete{
\nc{\mlabel}[1]{\label{#1}  
{\hfill \hspace{1cm}{\small\tt{{\ }\hfill(#1)}}}}
\nc{\mcite}[1]{\cite{#1}{\small{\tt{{\ }(#1)}}}}  
\nc{\mref}[1]{\ref{#1}{{\tt{{\ }(#1)}}}}  
\nc{\mbibitem}[1]{\bibitem[\bf #1]{#1}} 
}

\newtheorem{theorem}{Theorem}[section]

\newtheorem{lemma}[theorem]{Lemma}

\newtheorem{corollary}[theorem]{Corollary}
\theoremstyle{definition}

\newtheorem{claim}{Claim}[section]

\newtheorem{tempex}[theorem]{Example}
\newtheorem{tempexs}[theorem]{Examples}
\newtheorem{temprmk}[theorem]{Remark}
\newtheorem{tempexer}{Exercise}[section]

\nc{\tred}[1]{\textcolor{red}{#1}} \nc{\tgreen}[1]{\textcolor{green}{#1}}
\nc{\tblue}[1]{\textcolor{blue}{#1}} \nc{\tpurple}[1]{\textcolor{purple}{#1}}

\nc{\hu}[1]{\tpurple{\underline{Hu:}#1 }}
\nc{\xing}[1]{\tblue{\underline{Xing:}#1 }}

\nc{\GS}{Gr\"obner-Shirshov\xspace}
\nc{\gsb}{Gr\"{o}bner-Shirshov basis\xspace}
\nc{\gsbs}{Gr\"{o}bner-Shirshov bases\xspace}
\nc\olie{operated Lie algebra\xspace}
\nc\olies{operated Lie algebras\xspace}
\nc\bfone{\mathbf{1}}

\nc\nas[1]{{#1}^\ast}
\nc\alsw[1]{{\rm ALSW}(#1)}
\nc\nlsw[1]{{\rm NLSW}(#1)}
\nc\clie[1]{[#1]}
\nc\lbar[1]{\overline{#1}}
\nc\suba[1]{|_{#1}}
\nc\coplie[1]{\bfk\nlsbw{#1}}
\nc\coplieo[2]{\bfk{\rm NLSBW}_{#2}(#1)}
\nc\lb[1]{\left[#1\right]}\nc\dt[1]{{t^{(#1)}}}
\nc{\Irr}{\mathrm{Irr}}
\nc\blw[1]{\lfloor#1\rfloor}
\nc\plie[1]{\mathfrak{S}(#1)}
\nc\plien[1]{\mathcal{N}(#1)}
\nc{\lc}{\lfloor} \nc{\rc}{\rfloor}
\nc\Id{\rm Id}\nc\sopm[1]{\mathfrak{S}^\star(#1)}

\nc\ordc{>_{{\rm Dl}}} \nc\ordqc{\geq_{{\rm Dl}}}
\nc\ord{>_{{\rm IM }}}\nc\ordq{\geq _{\rm IM}}
\nc\ordd{>_{{\rm IM}}}\nc\ordqd{\geq_{{\rm IM}}}
\nc\ordb{>_{{\rm IM}}}\nc\ordqb{\geq_{{\rm IM}}}
\nc\alsbw[1]{{\rm ALSBW}_{\ordq}(#1)} \nc\nlsbw[1]{{\rm NLSBW}_{\ordq}(#1)}
\nc\alsbwo[2]{{\rm ALSBW}_{#2}(#1)} \nc\nlsbwo[2]{{\rm NLSBW}_{#2}(#1)}\nc\bws[1]{{\lfloor#1\rfloor}}\nc\oplie{{\rm OLie}(X)}
\nc{\dep}{{\rm dep}}
\nc\ordt{\geq_{\rm dt}}
\nc\ordtl{>_{\rm dt}}

\begin{document}

\title[Hook immanantal equalities  for linear combination matrices of (di)graphs]{Hook immanantal equalities  for linear combination matrices of (di)graphs and their applications}

\author{Xiangshuai Dong, Tingzeng Wu$^{*}$ and Hong-Jian Lai}\thanks{*Corresponding author}

\address{School of Mathematical Sciences, Xiamen University, Xiamen, Fujian 361005, China}
\email{a3566293588@163.com}

\address{School of Mathematics and Statistics, Qinghai Minzu University, Xining, Qinghai 810007, China}
\email{mathtzwu@163.com}

\address{School of Mathematics and System Sciences, Guangdong Polytechnic Normal University, Guangzhou 510665, China }
\email{hjlai2015@hotmail.com}

\date{\today}

\begin{abstract}
Let $\chi_\lambda$ be an irreducible character of the symmetric group $S_n$. For an $n \times n$ matrix $M = (m_{ij})$, define the immanant of $M$ corresponding to $\chi_\lambda$ by
\begin{eqnarray*}
d_\lambda(M) = \sum_{\sigma \in S_n} \chi_\lambda(\sigma) \prod_{i=1}^n m_{i\sigma(i)}.
\end{eqnarray*}
For $\lambda = (k, 1^{n-k})$, the immanant $d_{(k, 1^{n-k})}(M)$ is called the hook immanant and denoted by $d_k(M)$.
The hook immanant polynomial of matrix $M$ is defined as $d_{k}(xI_n - M)$, where $I_n$ is the $n \times n$ identity matrix. Let $G$ and $\overrightarrow{G}$ be a graph and a digraph, respectively. Suppose that $D(G)$ and $A(G)$ (resp. $D(\overrightarrow{G})$ and $A(\overrightarrow{G})$) are the degree matrix and  adjacency matrix  of $G$ (resp. $\overrightarrow{G}$), respectively. In this paper, we characterize two hook immanantal equalities for the linear combination of  matrices $\beta D(G)+\gamma A(G)$ and  $\beta D(\overrightarrow{G})+\gamma A(\overrightarrow{G})$, where $\beta$ and $\gamma$ are real numbers. As applications,  we  derive recursive formulas for the  hook immanantal polynomials  and hook immanants of graph matrices.
\end{abstract}

\makeatletter
\@namedef{subjclassname@2020}{\textup{2020} Mathematics Subject Classification}
\makeatother
\subjclass[2020]{
     05B20,
    05E05,
    05C31, 05C50, 15A15.
}

\keywords{Immanant; Hook immanant; Hook immanant polynomial; Linear combination  matrix; Graph matrix}

\maketitle

\tableofcontents

\setcounter{section}{0}

\allowdisplaybreaks

\section{Introduction}
Let $M = (m_{ij})$ be an $n \times n$ matrix. Let $S_n$ be the permutation group on $n$ symbols and $\chi_\lambda$ be an irreducible character of the symmetric group $S_n$, indexed by a partition $\lambda$ of $n$.  The {\em immanant function}, $d_\lambda$ associated with the character  $\chi_\lambda$ acting on $M$, is defined as
\begin{eqnarray*}
d_\lambda(M) = \sum_{\sigma \in S_n} \chi_\lambda(\sigma) \prod_{i=1}^n m_{i\sigma(i)}.
\end{eqnarray*}
For $\lambda = (k, 1^{n-k})$, $d_{(k, 1^{n-k})}(M)$ is called the {\em hook immanant} of $M$, denoted by $d_k(M)$. In particular, $d_1(M)=\text{det} M$ is the {\em  determinant} of $M$, $d_{2}(M)$ denotes the {\em  second immanant} of $M$ and $d_n(M) = \text{per} M$ is the {\em  permanent} of $M$. For the function $d_k(M)$ related to $k$, we stipulate that  if $k<1$ or $k>n$, then $d_k(M)=0$. Calculating the hook immanant is extremely challenging. B\"{u}rgisser \cite{bur} proved that computing the hook immanant of a matrix is VNP-complete. The hook immanant holds significant importance in algebraic graph theory, with profound connections to graph invariants and structural properties \cite{god, god2, liw}. Merris \cite{mer1, mer2} established a relationship between hook immanants and the number of hamiltonian cycles. Chan and Lam \cite{chan2, chan3} conducted systematic investigations into the Laplacian immanants of trees, determining the minimal values of Laplacian immanants for trees. The study of matrix immanants has been thoroughly elaborated in \cite{bol, gou, hai, li1, mar, ste} and other references. Moreover, immanants of graph matrices have garnered considerable attention, as seen in \cite{liu,mer3}, among others.

Let $I_n$ be an $n \times n$ identity matrix. The {\em hook immanant polynomial} $\Phi_k(M,x)$ (simply $\Phi_k(M)$) of an $n \times n$ matrix $M$ with afforded by  the character $\chi_{(k,1^{n-k})}$ is defined as $d_{(k,1^{n-k})}(xI_n - M)$, i.e.,
\begin{eqnarray*}
\Phi_k(M,x) = d_{(k,1^{n-k})}(xI_n - M).
\end{eqnarray*}
In particular, when $k = 1$, $\Phi_1(M,x)$ is called the {\em characteristic polynomial} of $M$, denoted by $\varphi(M,x)$. When $k = 2$, $\Phi_2(M,x)$ is called the {\em second immanant polynomial}. When $k = n$, $\Phi_n(M,x)$ is called the {\em permanental polynomial}, denoted by $\psi(M,x)$.
Merris \cite{bot}  stated that almost all trees possess a complete set of immanantal polynomials. Cash \cite{cas} extended Sachs' theorem to immanantal polynomials and applied this result to chemical structures containing hexagonal rings. Yu and Qu \cite{yu} utilized basic subgraphs to derive explicit expressions for the immanantal polynomials of matrix $M$. In algebraic graph theory, immanants corresponding to certain special partitions have been extensively investigated \cite{wu2, wu3, cvet, bor, liu}. For further research on immanantal polynomials, refer to \cite{bro, mer4, nag}.

We are interested in conducting ongoing research into the properties of the hook immanant and hook immanantal polynomial of graph matrices. Below, we introduce some relevant notations.

Let $G = (V(G), E(G))$ be a simple graph with vertex set $ V(G) = \{v_1, v_2, \ldots, v_n\} $ and edge set $ E(G) $. Let $ d(v_i) $ be the degree of vertex $ v_i $ in $ G $. The degree matrix of $ G $ is denoted by $ D(G) = \mathrm{diag}(d(v_1), d(v_2), \ldots, d(v_n)) $. The adjacency matrix of $ G $ is an $ n \times n $ matrix, denoted by $ A(G) = (a_{ij}) $, whose entry $a_{ij}$ is given by
$$a_{ij} =
\begin{cases}
1, & \text{if } v_i \text{ is adjacent to } v_j, \\
0, & \text{otherwise}.
\end{cases}$$
Let $\overrightarrow{G} = (V(\overrightarrow{G}), E(\overrightarrow{G}))$ be a digraph without loops or parallel arcs, where the vertex set is $V(\overrightarrow{G}) = \{v_{1}, v_{2}, \ldots, v_{n}\}$ and the arc set is $E(\overrightarrow{G})$. The adjacency matrix of  $\overrightarrow{G}$ is defined as $A(\overrightarrow{G}) = (a_{ij})_{n \times n}$, where
$$a_{ij} =
\begin{cases}
1, & \text{if } (v_i, v_j) \in E(\overrightarrow{G}), \\
0, & \text{otherwise}.
\end{cases}$$
The out-degree diagonal matrix is denoted by $D(\overrightarrow{G}) = \mathrm{diag}(d^{+}(v_{1}), d^{+}(v_{2}), \ldots, d^{+}(v_{n}))$, where $d^{+}(v_{i})$ denotes the out-degree of the vertex $v_{i}$.

Let $H(G) = \beta D(G) + \gamma A(G)$ and $H(\overrightarrow{G}) = \beta D(\overrightarrow{G}) + \gamma A(\overrightarrow{G})$, where $\beta$ and $\gamma$ are real numbers. If $\beta=0$ and $\gamma=1$, then $H(G) = A(G)$ (resp. $H(\overrightarrow{G}) = A(\overrightarrow{G})$) is the adjacency matrix of graph $G$ (resp. $\overrightarrow{G}$). If  $\beta=1$ and $\gamma=-1$, then $H(G) = D(G)-A(G)=L(G)$ (resp. $H(\overrightarrow{G}) = D(\overrightarrow{G})-A(\overrightarrow{G})=L(\overrightarrow{G})$) is the Laplacian matrix of graph $G$ (resp. $\overrightarrow{G}$). If $\beta=1$ and $\gamma=1$, then $H(G) = D(G)+A(G)=Q(G)$ (resp. $H(\overrightarrow{G}) = D(\overrightarrow{G})+A(\overrightarrow{G})=Q(\overrightarrow{G})$) is the signless Laplacian matrix of graph $G$ (resp. $\overrightarrow{G}$). Suppose $0\leq\alpha\leq1$. If $\beta=\alpha$ and $\gamma=1-\alpha$, then $H(G) = \alpha D(G)+(1-\alpha)A(G)=A_{a}(G)$ (resp. $H(\overrightarrow{G}) = \alpha D(\overrightarrow{G})+(1-\alpha)A(\overrightarrow{G})=A_{a}(\overrightarrow{G})$) is the $A_{a}$ matrix of graph $G$ (resp. $\overrightarrow{G}$).

Let $S \subseteq V(G)$ be a vertex subset. Assume that $H_S(G)$ is the principal submatrix of $H(G)$ obtained by deleting the rows and columns corresponding to all vertices in $S \subseteq V(G)$. In particular, if $S = \{v\}$ and $v \in V(G)$, then $H_{\{v\}}(G)$ is simply written as $H_v(G)$.
Similarly, if $S = \{u,v\}$  where $u, v \in V(G)$, then 
is abbreviated as  $H_{uv}(G)$. Suppose that $G-S$ denotes the graph obtained by deleting all vertices in $S$ and all  edges incident with the vertices in $S$ from $G$. In particular, if $S = \{v\}$ and $v \in V(G)$, $G-\{v\}$ is simply written as $G-v$. The graph obtained by deleting an edge from $G$ is denoted by $ G - e $. If $v$ and $\overrightarrow{e}$ are a vertex and an arc of $\overrightarrow{G}$ respectively, ${\overrightarrow {G}} - v$ is the graph obtained by deleting vertex $v$ and the arcs incident to $v$ from $\overrightarrow {G}$. ${\overrightarrow {G}}-{\overrightarrow {e}}$ is obtained by deleting $\overrightarrow{e}$ but retaining all vertices and other  arcs. The set of vertices adjacent to $v \in V(G)$ is called the {\em neighborhood} of $v$, denoted by $N(v)$. Let $\mathfrak{C}_G(v)$ and $\mathfrak{C}_G(e)$ be the sets of cycles in $G$ containing vertex $v \in V(G)$ and edge $e \in E(G)$, respectively. If a directed cycle has a consistent direction, it is called a {\em consistently directed cycle}, denoted by $\overrightarrow{C}$. Let $\mathcal{C}_{\overrightarrow {G}}(v)$ and $\mathcal{C}_{\overrightarrow {G}}(e)$ be the sets of consistently directed cycles in $G$ containing vertex $v \in V(G)$ and edge $e \in E(G)$, respectively.

Next, we mainly characterize recursive formulas for the hook immanant polynomials  of  $H(G)$ and $H(\overrightarrow{G})$.

\begin{theorem}\label{thm1.1}
$(i)$ Let $v$ be an arbitrarily given vertex of  $G$. Then
\begin{eqnarray*}
\Phi_k(H(G),x)
&=& (x-\beta d(v))[\Phi_{k-1}(H_v(G)) + \Phi_k(H_v(G))] \\
&& + \gamma^2 \sum_{u\in N(v)} [\Phi_{k-2}(H_{uv}(G)) - \Phi_k(H_{uv}(G))] \\
&& +2 \sum_{C\in \mathfrak{C}_G(v)} \gamma^{|V(C)|}[(-1)^{|V(C)|}\Phi_{k-|V(C)|}(H_{V(C)}(G)) - \Phi_k(H_{V(C)}(G))].
\end{eqnarray*}

$(ii)$ Let $e=uv$  be an arbitrarily given edge of  $G$. Then
\begin{eqnarray*}
\Phi_k(H(G),x)
&=& \Phi_k(H(G-e),x) - \beta[\Phi_k(H_v(G-e)) + \Phi_{k-1}(H_v(G-e))] \\
&& - \beta[\Phi_k(H_u(G-e)) + \Phi_{k-1}(H_u(G-e))] \\
&& + [(\beta^2 - \gamma^2)\Phi_k(H_{uv}(G)) + 2\beta^2 \Phi_{k-1}(H_{uv}(G)) + (\beta^2 + \gamma^2)\Phi_{k-2}(H_{uv}(G))] \\
&& +2 \sum_{C\in \mathfrak{C}_G(e))} \gamma^{|V(C)|}[(-1)^{|V(C)|}\Phi_{k-|V(C)|}(H_{V(C)}(G)) - \Phi_k(H_{V(C)}(G))].
\end{eqnarray*}
\end{theorem}
The hook immanant  polynomials for trees (a specialized class of graphs) have the following recursive formulation:
\begin{corollary}
$(i)$ Let $v$ be an arbitrarily given vertex of  tree $T$. Then
\begin{eqnarray*}
\Phi_k(H(T),x)
&=& (x-\beta d(v))[\Phi_{k-1}(H_v(T)) + \Phi_k(H_v(T))] \\
&& + \gamma^2 \sum_{u\in N(v)} [\Phi_{k-2}(H_{uv}(T)) - \Phi_k(H_{uv}(T))].
\end{eqnarray*}

$(ii)$ Let $e=uv$ be an arbitrarily given edge of  tree $T$. Then
\begin{eqnarray*}
\Phi_k(H(T),x)
&=& \Phi_k(H(T-e),x) - \beta[\Phi_k(H_v(T-e)) + \Phi_{k-1}(H_v(T-e))] \\
&& - \beta[\Phi_k(H_u(T-e)) + \Phi_{k-1}(H_u(T-e))] \\
&& + [(\beta^2 - \gamma^2)\Phi_k(H_{uv}(T)) + 2\beta^2 \Phi_{k-1}(H_{uv}(T)) + (\beta^2 + \gamma^2)\Phi_{k-2}(H_{uv}(T))].
\end{eqnarray*}
\end{corollary}

\begin{theorem}\label{thm1.3}
$(i)$ Let $v$ be an arbitrarily given vertex of  $\overrightarrow{G}$. Then
\begin{eqnarray*}
\Phi_k(H(\overrightarrow{G}),x)
&=& (x-\beta d^{+}(v))[\Phi_{k-1}(H_v(\overrightarrow{G})) + \Phi_k(H_v(\overrightarrow{G}))] \\
&& +\sum_{\overrightarrow{C}\in \mathcal{C}_{\overrightarrow{G}}(v)} \gamma^{|V(\overrightarrow{C})|}[(-1)^{|V(\overrightarrow{C})|}\Phi_{k-|V(\overrightarrow{C})|}(H_{V(\overrightarrow{C})}(\overrightarrow{G})) - \Phi_k(H_{V(\overrightarrow{C})}(\overrightarrow{G}))].
\end{eqnarray*}

$(ii)$ Let $\overrightarrow{e}=(v,u)$ be an arbitrarily given arc of  $\overrightarrow{G}$. Then
\begin{eqnarray*}
\Phi_k(H(\overrightarrow{G}),x)
&=& \Phi_k(H(\overrightarrow{G}-\overrightarrow{e}),x) - \beta[\Phi_k(H_v(\overrightarrow{G}-\overrightarrow{e})) + \Phi_{k-1}(H_v(\overrightarrow{G}-\overrightarrow{e}))] \\
&& + \sum_{\overrightarrow{C}\in \mathcal{C}_{\overrightarrow{G}}(e)} \gamma^{|V(\overrightarrow{C})|}[(-1)^{|V(\overrightarrow{C})|}\Phi_{k-|V(\overrightarrow{C})|}(H_{V(\overrightarrow{C})}(\overrightarrow{G})) - \Phi_k(H_{V(\overrightarrow{C})}(\overrightarrow{G}))].
\end{eqnarray*}
\end{theorem}
Note that $d_k(H(G))=(-1)^{|V(G)|}\Phi_k(H(G),0)$ (resp. $d_k(H(\overrightarrow{G}))=(-1)^{|V(\overrightarrow{G})|}\Phi_k(H(\overrightarrow{G}),0)$) for a graph $G$ (resp. $\overrightarrow{G}$). Then the recursive formulas for  hook immanant of  $H(G)$ and $H(\overrightarrow{G})$ follow directly from  Theorems \ref{thm1.1} and \ref{thm1.3}, respectively.
\begin{theorem}\label{thm1.2}
$(i)$ Let $v$ be an arbitrarily given vertex of  $G$. Then
\begin{eqnarray*}
d_k(H(G))
&=& \beta d(v)[d_{k-1}(H_v(G)) + d_k(H_v(G))] \\
&& + \gamma^2 \sum_{u\in N(v)} [d_{k-2}(H_{uv}(G)) - d_k(H_{uv}(G))] \\
&& +2 \sum_{C\in \mathfrak{C}_G(v)} \gamma^{|V(C)|}[d_{k-|V(C)|}(H_{V(C)}(G)) -(-1)^{|V(C)|}d_k(H_{V(C)}(G))].
\end{eqnarray*}

$(ii)$ Let $e=uv$ be an arbitrarily given edge of  $G$. Then
\begin{eqnarray*}
d_k(H(G))
&=& d_k(H(G-e)) + \beta[d_k(H_v(G-e)) + d_{k-1}(H_v(G-e))] \\
&& + \beta[d_k(H_u(G-e)) + d_{k-1}(H_u(G-e))] \\
&& +[(\beta^2 - \gamma^2)d_k(H_{uv}(G)) + 2\beta^2 d_{k-1}(H_{uv}(G)) + (\beta^2 + \gamma^2)d_{k-2}(H_{uv}(G))] \\
&& +2 \sum_{C\in \mathfrak{C}_G(e)} \gamma^{|V(C)|}[d_{k-|V(C)|}(H_{V(C)}(G)) - (-1)^{|V(C)|}d_k(H_{V(C)}(G))].
\end{eqnarray*}
\end{theorem}
For trees, a special class of graphs, the recursive formula for their hook immanant is as follows:
\begin{corollary}\label{cor1.1}
$(i)$ Let $v$ be an arbitrarily given vertex of  tree $T$. Then
\begin{eqnarray*}
d_k(H(T))
&=& \beta d(v)[d_{k-1}(H_v(T)) + d_k(H_v(T))] \\
&& + \gamma^2 \sum_{u\in N(v)} [d_{k-2}(H_{uv}(T)) - d_k(H_{uv}(T))].
\end{eqnarray*}

$(ii)$ Let $e=uv$ be an arbitrarily given edge of  tree $T$. Then
\begin{eqnarray*}
d_k(H(T))
&=& d_k(H(T-e)) + \beta[d_k(H_v(T-e)) + d_{k-1}(H_v(T-e))] \\
&& + \beta[d_k(H_u(T-e)) + d_{k-1}(H_u(T-e))] \\
&& +[(\beta^2 - \gamma^2)d_k(H_{uv}(T)) + 2\beta^2 d_{k-1}(H_{uv}(T)) + (\beta^2 + \gamma^2)d_{k-2}(H_{uv}(T))].
\end{eqnarray*}
\end{corollary}

\begin{theorem}\label{thm1.4}
$(i)$ Let $v$ be an arbitrarily given vertex of  $\overrightarrow{G}$. Then
\begin{eqnarray*}
d_k(H(G))
&=& \beta d^{+}(v)[d_{k-1}(H_v(\overrightarrow{G})) + d_k(H_v(\overrightarrow{G}))] \\
&& + \sum_{\overrightarrow{C}\in \mathcal{C}_{\overrightarrow{G}}(v)} \gamma^{|V(\overrightarrow{C})|}[d_{k-|V(C)|}(H_{V(\overrightarrow{C})}(\overrightarrow{G}))-(-1)^{|V(\overrightarrow{C})|} d_k(H_{V(\overrightarrow{C})}(\overrightarrow{G}))].
\end{eqnarray*}

$(ii)$ Let $\overrightarrow{e}=(v,u)$ be an arbitrarily given arc of  $\overrightarrow{G}$. Then
\begin{eqnarray*}
d_k(H(\overrightarrow{G}))
&=& d_k(H(\overrightarrow{G}-\overrightarrow{e})) + \beta[d_k(H_v(\overrightarrow{G}-\overrightarrow{e})) + d_{k-1}(H_v(\overrightarrow{G}-\overrightarrow{e}))] \\
&& + \sum_{\overrightarrow{C}\in \mathcal{C}_{\overrightarrow{G}}(e)} \gamma^{|V(\overrightarrow{C})|}[d_{k-|V(\overrightarrow{C})|}(H_{V(\overrightarrow{C})}(\overrightarrow{G}))-(-1)^{|V(\overrightarrow{C})|} d_k(H_{V(\overrightarrow{C})}(\overrightarrow{G}))].
\end{eqnarray*}
\end{theorem}

The organization of this paper is as follows: In Section 2,  we prove Theorem \ref{thm1.1}. In Section 3, we prove Theorem \ref{thm1.3}. Applications of Theorems \ref{thm1.1} and \ref{thm1.2} are given in the last section.

\section{Proof of Theorem \ref{thm1.1}}

Before proving  $(i)$ of Theorem \ref{thm1.1}, we first introduce the famous Murnaghan-Nakayama rule.

\begin{theorem}(Sagan, \cite{26})\label{lem2.1}
Let $\lambda$ be a partition and $\alpha = (\alpha_1, \ldots, \alpha_k)$ be a composition. Then
$$
\chi_{\alpha}^{\lambda} = \sum_{\xi} (-1)^{h(\xi)}\chi_{\alpha/\alpha_1}^{\lambda/\xi}
$$
where the sum runs over all $rim\ hooks\ \xi$ of length $\alpha_1$ in $\lambda$, and $h(\xi) = (\text{the number of rows of } \xi) - 1$.
\end{theorem}

Next, we will introduce the linear properties of hook immanant calculation, which is essential for the subsequent proofs. Let $M_1$, $M_2$ and $M_3$ be three $n \times n$ matrices defined as follows:
$$M_1 =
\begin{bmatrix}
m_{11} & m_{12} & \cdots & m_{1n} \\
\vdots & \vdots & \ddots & \vdots \\
b_{i1} + c_{i1} & b_{i2} + c_{i2} & \cdots & b_{in} + c_{in} \\
\vdots & \vdots & \ddots & \vdots \\
m_{n1} & m_{n2} & \cdots & m_{nn}
\end{bmatrix},$$
$$M_2 =
\begin{bmatrix}
m_{11} & m_{12} & \cdots & m_{1n} \\
\vdots & \vdots & \ddots & \vdots \\
b_{i1} & b_{i2} & \cdots & b_{in} \\
\vdots & \vdots & \ddots & \vdots \\
m_{n1} & m_{n2} & \cdots & m_{nn}
\end{bmatrix}
~{\rm and}~
M_3 =
\begin{bmatrix}
m_{11} & m_{12} & \cdots & m_{1n} \\
\vdots & \vdots & \ddots & \vdots \\
c_{i1} & c_{i2} & \cdots & c_{in} \\
\vdots & \vdots & \ddots & \vdots \\
m_{n1} & m_{n2} & \cdots & m_{nn}
\end{bmatrix}.$$

\begin{lemma}\label{lem3.2}
Suppose that $M_1$, $M_2$ and $M_3$ are the three matrices defined above. Then
$$d_\lambda(M_1) = d_\lambda(M_2) + d_\lambda(M_3).$$
\end{lemma}

\begin{proof}
As $d_\lambda(M_1)= \sum\limits_{\sigma\in S_n} \chi_\lambda(\sigma)m_{1\sigma(1)}m_{2\sigma(2)} \cdots m_{n\sigma(n)}$, the expansion of $d_\lambda(M_1)$ consists of $n!$ terms. We divide these $n!$ terms into $n$ groups: $\mathfrak{M}_1$, $\mathfrak{M}_2$,\ldots,$\mathfrak{M}_n$, where for $j\in{1,2,\ldots,n}$, $\mathfrak{M}_j$ contains terms involving $b_{ij} + c_{ij}$. Define $\mathfrak{P}_1 = \{\sigma|m_{1\sigma(1)}m_{2\sigma(2)} \cdots m_{n\sigma(n)} \in \mathfrak{M}_1\}$, $\mathfrak{P}_2 = \{\sigma|m_{1\sigma(1)}m_{2\sigma(2)} \cdots m_{n\sigma(n)} \in \mathfrak{M}_2\}$, $\ldots$, $\mathfrak{P}_n = \{\sigma|m_{1\sigma(1)}m_{2\sigma(2)} \cdots m_{n\sigma(n)} \in \mathfrak{M}_n\}$. Obviously, $\mathfrak{P}_1 \cup \mathfrak{P}_2 \cup \cdots \cup \mathfrak{P}_n = S_n$. Hence,
\begin{eqnarray*}
&&d_\lambda(M_1)\\
&=& \sum_{\sigma\in S_n} \chi_\lambda(\sigma)m_{1\sigma(1)}m_{2\sigma(2)} \cdots m_{n\sigma(n)} \\
&=& \sum_{\sigma\in \mathfrak{P}_1} \chi_\lambda(\sigma)m_{1\sigma(1)}m_{2\sigma(2)} \cdots m_{n\sigma(n)} + \sum_{\sigma\in \mathfrak{P}_2} \chi_\lambda(\sigma)m_{1\sigma(1)}m_{2\sigma(2)} \cdots m_{n\sigma(n)} \\
&&+ \cdots + \sum_{\sigma\in \mathfrak{P}_n} \chi_\lambda(\sigma)m_{1\sigma(1)}m_{2\sigma(2)} \cdots m_{n\sigma(n)} \\
&=& \sum_{\sigma\in \mathfrak{P}_1} \chi_\lambda(\sigma)(b_{i1} + c_{i1})m_{1\sigma(1)} \cdots m_{i-1\sigma(i-1)}m_{i+1\sigma(i+1)} \cdots m_{n\sigma(n)} \\
&& + \sum_{\sigma\in \mathfrak{P}_2} \chi_\lambda(\sigma)(b_{i2} + c_{i2})m_{1\sigma(1)} \cdots m_{i-1\sigma(i-1)}m_{i+1\sigma(i+1)} \cdots m_{n\sigma(n)} \\
&& + \cdots + \sum_{\sigma\in \mathfrak{P}_n} \chi_\lambda(\sigma)(b_{in} + c_{in})m_{1\sigma(1)} \cdots m_{i-1\sigma(i-1)}m_{i+1\sigma(i+1)} \cdots m_{n\sigma(n)} \\
&=& \sum_{\sigma\in \mathfrak{P}_1} \chi_\lambda(\sigma)b_{i1}m_{1\sigma(1)} \cdots m_{i-1\sigma(i-1)}m_{i+1\sigma(i+1)} \cdots m_{n\sigma(n)}\\
&& + \sum_{\sigma\in \mathfrak{P}_1} \chi_\lambda(\sigma)c_{i1}m_{1\sigma(1)} \cdots m_{i-1\sigma(i-1)}m_{i+1\sigma(i+1)} \cdots m_{n\sigma(n)}\\
&& + \sum_{\sigma\in \mathfrak{P}_2} \chi_\lambda(\sigma)b_{i2}m_{1\sigma(1)} \cdots m_{i-1\sigma(i-1)}m_{i+1\sigma(i+1)} \cdots m_{n\sigma(n)}\\
&& + \sum_{\sigma\in \mathfrak{P}_2} \chi_\lambda(\sigma)c_{i2}m_{1\sigma(1)} \cdots m_{i-1\sigma(i-1)}m_{i+1\sigma(i+1)} \cdots m_{n\sigma(n)}\\
&& + \cdots + \sum_{\sigma\in \mathfrak{P}_n} \chi_\lambda(\sigma)b_{in}m_{1\sigma(1)} \cdots m_{i-1\sigma(i-1)}m_{i+1\sigma(i+1)} \cdots m_{n\sigma(n)} \\
&& + \sum_{\sigma\in \mathfrak{P}_n} \chi_\lambda(\sigma)c_{in}m_{1\sigma(1)} \cdots m_{i-1\sigma(i-1)}m_{i+1\sigma(i+1)} \cdots m_{n\sigma(n)}\\
&=& \Bigg[ \sum_{\sigma\in\mathfrak{P}_1} \chi_\lambda(\sigma)b_{i1}m_{1\sigma(1)} \cdots m_{i-1\sigma(i-1)}m_{i+1\sigma(i+1)} \cdots m_{n\sigma(n)}  \\
&& + \sum_{\sigma\in \mathfrak{P}_2} \chi_\lambda(\sigma)b_{i2}m_{1\sigma(1)} \cdots m_{i-1\sigma(i-1)}m_{i+1\sigma(i+1)} \cdots m_{n\sigma(n)}\\
&& + \cdots +  \sum_{\sigma\in \mathfrak{P}_n} \chi_\lambda(\sigma)b_{in}m_{1\sigma(1)} \cdots m_{i-1\sigma(i-1)}m_{i+1\sigma(i+1)} \cdots m_{n\sigma(n)} \Bigg] \\
&& +\Bigg[ \sum_{\sigma\in \mathfrak{P}_1} \chi_\lambda(\sigma)c_{i1}m_{1\sigma(1)} \cdots m_{i-1\sigma(i-1)}m_{i+1\sigma(i+1)} \cdots m_{n\sigma(n)}  \\
&& + \sum_{\sigma\in \mathfrak{P}_2} \chi_\lambda(\sigma)c_{i2}m_{1\sigma(1)} \cdots m_{i-1\sigma(i-1)}m_{i+1\sigma(i+1)} \cdots m_{n\sigma(n)}\\
&& + \cdots +  \sum_{\sigma\in \mathfrak{P}_n} \chi_\lambda(\sigma)c_{in}m_{1\sigma(1)} \cdots m_{i-1\sigma(i-1)}m_{i+1\sigma(i+1)} \cdots m_{n\sigma(n)} \Bigg] \\
&=& d_\lambda(M_2)+d_\lambda(M_3).
\end{eqnarray*}
\end{proof}

 \textbf{Proof of  Theorem \ref{thm1.1} .} $(i)$
Denote  $V(G)=\{v_1, v_2, \ldots, v_n\}$. Index the rows and columns of $H(G)$ by $v_1, v_2, \ldots, v_n$.
 Consider a term $a_{1\sigma(1)}a_{2\sigma(2)}\cdots a_{n\sigma(n)}$ in the expansion of $d_k(xI_n - \beta D(G) - \gamma A(G))$, where $a_{ij}$ is the $(i,j)$-entry of $xI_n - \beta D(G) - \gamma A(G)$. Obviously, $a_{ii} = x - \beta d(v_i)$, and when $i\neq j$, if $v_i v_j \in E(G)$, then $a_{ij} = -\gamma$; otherwise it is $0$. Therefore, if $a_{1\sigma(1)}a_{2\sigma(2)}\cdots a_{n\sigma(n)} \neq 0$, then either $j=\sigma(j)$ or $v_j v_{\sigma(j)}\in E(G)$.  Define $\mathfrak{P} = \{\sigma|a_{1\sigma(1)}a_{2\sigma(2)}\cdots a_{n\sigma(n)} \neq 0\}$. In the following, we only consider such $\sigma\in\mathfrak{P}$. Without loss of generality, let $a_{11} = x - \beta d(v)$ correspond to the vertex $v$. Note that each $\sigma$ can be expressed as a product of disjoint permutation cycles. Then we write $\sigma = \gamma_1 \sigma'$, where $\gamma_1$ is a permutation cycle. Depending on the length of $\gamma_1$, $\sigma$ can be divided  into three cases:
\begin{eqnarray*}
&&S_{1}=\{\sigma\in \mathfrak{P}, \gamma_{1}=(1), \text{i.e., }~\sigma~\text{fixes 1}\},\\
&&S_{2}^{j}=\{\sigma\in \mathfrak{P}, \gamma_{1}=(1j), \text{i.e., }~\gamma_{1}~\text{corresponds to an edge }~vv_{j}~\text{in }~G\},\\
&&S_{C}=\{\sigma\in \mathfrak{P}, \gamma_{1}~\text{corresponds to a  cycle }~C~\text{of length }~l\geq 3~\text{which  contains the vertex }~v\\
&&\text{of }~G\}.
\end{eqnarray*}
By Theorem \ref{lem2.1}, the following holds:\\
$(a)$. If $\sigma \in S_1$, then
\begin{eqnarray}\label{equ2.1.1}
\chi_{(k,1^{n-k})}(\sigma) = \chi_{(k-1,1^{n-k})}(\sigma') + \chi_{(k,1^{n-k-1})}(\sigma'),
\end{eqnarray}
$(b)$. If $\sigma \in S_2^{j}$, then
\begin{eqnarray}\label{equ2.1.2}
\chi_{(k,1^{n-k})}(\sigma) = \chi_{(k-2,1^{n-k})}(\sigma') - \chi_{(k,1^{n-k-2})}(\sigma'),
\end{eqnarray}
$(c)$. If $\sigma \in S_C$, then
\begin{eqnarray}\label{equ2.1.3}
\chi_{(k,1^{n-k})}(\sigma) = \chi_{(k-l,1^{n-k})}(\sigma') + (-1)^{l-1}\chi_{(k,1^{n-k-l})}(\sigma').
\end{eqnarray}
Thus,
\begin{eqnarray}\label{equ2.1}
&&d_k(xI_n - \beta D(G) - \gamma A(G)) \nonumber\\
&=& \sum_{\sigma} \chi_{(k,1^{n-k})}(\sigma)a_{1\sigma(1)}a_{2\sigma(2)}\cdots a_{n\sigma(n)}\nonumber \\
&=& \sum_{\sigma \in S_1} \chi_{(k,1^{n-k})}(\sigma)a_{1\sigma(1)}a_{2\sigma(2)}\cdots a_{n\sigma(n)}+ \sum_j \sum_{\sigma \in S_2^j} \chi_{(k,1^{n-k})}(\sigma)a_{1\sigma(1)}a_{2\sigma(2)}\cdots a_{n\sigma(n)} \nonumber\\
&& + \sum_C \sum_{\sigma \in S_C} \chi_{(k,1^{n-k})}(\sigma)a_{1\sigma(1)}a_{2\sigma(2)}\cdots a_{n\sigma(n)} \nonumber\\
&=& (x-\beta d(v)) \sum_{\sigma' \in S_1} [\chi_{(k-1,1^{n-k})}(\sigma')+\chi_{(k,1^{n-k-1})}(\sigma')]a_{2\sigma'(2)} \cdots a_{n\sigma'(n)}\nonumber \\
&& + \gamma^2 \sum_{v v_j \in E(G)} \sum_{\sigma' \in S_2^j} [\chi_{(k-2,1^{n-k})}(\sigma')-\chi_{(k,1^{n-k-2})}(\sigma')]a_{2\sigma'(2)} \cdots a_{(j-1)\sigma'(j-1)} \nonumber\\
&& \cdot a_{(j+1)\sigma'(j+1)} \cdots a_{n\sigma'(n)} \\
&& + 2 \sum_{C \in \mathfrak{C}_G(v)} (- \gamma)^{l} \sum_{\sigma' \in S_C} [\chi_{(k-l,1^{n-k})}(\sigma')+(-1)^{l-1}\chi_{(k,1^{n-k-l})}(\sigma')]a_{i_1\sigma^{'}(i_1)} \nonumber\\
&&\cdot a_{i_2\sigma^{'}(i_2)} \cdots a_{i_{n-l}\sigma^{'}(i_{n-l})}\nonumber \\
&=& (x-\beta d(v)) \sum_{\sigma' \in S_1} [\chi_{(k-1,1^{n-k})}(\sigma')a_{2\sigma'(2)} \cdots a_{n\sigma'(n)} \nonumber\\
&& + \chi_{(k,1^{n-k-1})}(\sigma')a_{2\sigma'(2)} \cdots a_{n\sigma'(n)}]\nonumber \\
&& + \gamma^2 \sum_{vv_j \in E(G)} \sum_{\sigma' \in S_2^j} [\chi_{(k-2,1^{n-k})}(\sigma')a_{2\sigma'(2)} \cdots a_{(j-1)\sigma'(j-1)}a_{(j+1)\sigma'(j+1)} \cdots a_{n\sigma'(n)}\nonumber \\
&& - \chi_{(k,1^{n-k-2})}(\sigma')a_{2\sigma'(2)} \cdots a_{(j-1)\sigma'(j-1)}a_{(j+1)\sigma'(j+1)} \cdots a_{n\sigma'(n)}]\nonumber \\
&& + 2 \sum_{C \in \mathfrak{C}_G(v)} (- \gamma)^{l} \sum_{\sigma' \in S_C} [\chi_{(k-l,1^{n-k})}(\sigma')a_{i_1\sigma^{'}(i_1)}a_{i_2\sigma^{'}(i_2)} \cdots a_{i_{n-l}\sigma^{'}(i_{n-l})} \nonumber\\
&& + (-1)^{l-1}\chi_{(k,1^{n-k-l})}(\sigma')a_{i_1\sigma^{'}(i_1)}a_{i_2\sigma^{'}(i_2)} \cdots a_{i_{n-l}\sigma^{'}(i_{n-l})}] \nonumber\\
&=& (x-\beta d(v))[\Phi_{k-1}(H_v(G))+\Phi_k(H_v(G))]+\gamma^2 \sum_{vv_j \in E(G)} [\Phi_{k-2}(H_{v_j v}(G))-\Phi_k(H_{v_j v}(G))] \nonumber\\
&& + 2 \sum_{C \in \mathfrak{C}_G(v)} \gamma^{|V(C)|}[(-1)^{|V(C)|}\Phi_{k-|V(C)|}(H_{V(C)}(G))-\Phi_k(H_{V(C)}(G))],\nonumber
\end{eqnarray}
where in Formula (\ref{equ2.1}), $i_1,\ldots,i_{n-l}$ are indices of the vertices in the graph $G-V(C)$. 

$(ii)$ Applying $(i)$ of Theorem \ref{thm1.1}  to both $G$ and $G-e$, respectively, we obtain
\begin{eqnarray*}
\Phi_k(H(G),x)
&=& (x-\beta d(v))[\Phi_{k-1}(H_v(G)) + \Phi_k(H_v(G))] \\
&& + \gamma^2 \sum_{vw\in E(G)} [\Phi_{k-2}(H_{vw}(G)) - \Phi_k(H_{vw}(G))] \\
&& +2 \sum_{C\in \mathfrak{C}_G(v)} \gamma^{|V(C)|}[(-1)^{|V(C)|}\Phi_{k-l}(H_{V(C)}(G)) - \Phi_k(H_{V(C)}(G))]
\end{eqnarray*}
and
\begin{eqnarray*}
\Phi_k(H(G-e),x)
&=& (x-\beta (d(v)-1))[\Phi_{k-1}(H_v(G-e)) + \Phi_k(H_v(G-e))] \\
&& + \gamma^2 \sum_{vw\in E(G-e)} [\Phi_{k-2}(H_{vw}(G-e)) - \Phi_k(H_{vw}(G-e))] \\
&& +2 \sum_{C\in \mathfrak{C}_{G-e}(v)} \gamma^{|V(C)|}[(-1)^{|V(C)|}\Phi_{k-l}(H_{V(C)}(G-e)) - \Phi_k(H_{V(C)}(G-e))].
\end{eqnarray*}
Then
\begin{eqnarray}\label{equ3.2}
&&\Phi_k(H(G),x)-\Phi_k(H(G-e),x)\nonumber\\
&=& \Bigg\{(x-\beta d(v))[\Phi_{k-1}(H_v(G)) + \Phi_k(H_v(G))]\nonumber\\
&&-(x-\beta (d(v)-1))[\Phi_{k-1}(H_v(G-e)) + \Phi_k(H_v(G-e))]\Bigg\} \nonumber\\
&& + \gamma^2 \Bigg\{\sum_{vw\in E(G)} [\Phi_{k-2}(H_{uv}(G)) - \Phi_k(H_{uv}(G))]\\
&&-\sum_{vw\in E(G-e)} [\Phi_{k-2}(H_{uv}(G-e)) - \Phi_k(H_{uv}(G-e))]\Bigg\}\nonumber \\
&& +2 \Bigg\{\sum_{C\in \mathfrak{C}_G(v)} \gamma^{|V(C)|}[(-1)^{|V(C)|}\Phi_{k-1}(H_{V(C)}(G))\nonumber\\
&&- \Phi_k(H_{V(C)}(G))]-\sum_{C\in \mathfrak{C}_{G-e}(v)} \gamma^{|V(C)|}[(-1)^{|V(C)|}\Phi_{k-1}(H_{V(C)}(G-e)) - \Phi_k(H_{V(C)}(G-e))]\Bigg\}.\nonumber
\end{eqnarray}

By (\ref{equ2.1.1}), (\ref{equ2.1.2}), (\ref{equ2.1.3}) and Lemma \ref{lem3.2}, we conclude that
\begin{eqnarray}\label{equ3.3}
\Phi_{k-1}(H_v(G))=\Phi_{k-1}(H_v(G-e))-\beta \Phi_{k-1}(H_{uv}(G))-\beta \Phi_{k-2}(H_{uv}(G)),
\end{eqnarray}
\begin{eqnarray}\label{equ3.4}
\Phi_{k}(H_v(G))=\Phi_{k}(H_v(G-e))-\beta \Phi_k(H_{uv}(G))-\beta \Phi_{k-1}(H_{uv}(G)),
\end{eqnarray}
\begin{eqnarray}\label{equ3.5}
\sum_{vw\in E(G)\atop w\neq u}  \Phi_k(H_{vw}(G))&=&\sum_{vw\in E(G-e)} \Phi_k(H_{vw}(G-e))-\beta \sum_{vw\in E(G)\atop w\neq u}  \Phi_k(H_{\{v,w,u\}}(G))\\
&&-\beta \sum_{vw\in E(G)\atop w\neq u}  \Phi_{k-1}(H_{\{v,w,u\}}(G)),\nonumber
\end{eqnarray}
\begin{eqnarray}\label{equ3.5.1}
\sum_{vw\in E(G)\atop w\neq u}  \Phi_{k-2}(H_{vw}(G))&=&\sum_{vw\in E(G-e)} \Phi_{k-2}(H_{vw}(G-e))-\beta \sum_{vw\in E(G)\atop w\neq u}  \Phi_{k-2}(H_{\{v,w,u\}}(G))\nonumber\\
&&-\beta \sum_{vw\in E(G)\atop w\neq u}  \Phi_{k-3}(H_{\{v,w,u\}}(G)),
\end{eqnarray}
\begin{eqnarray}\label{equ3.6}
&&\sum_{C\in \mathfrak{C}_G(v)\atop u\notin V(C)} [\gamma^{|V(C)|}(-1)^{|V(C)|}\Phi_{k-l}(H_{V(C)}(G))\nonumber\\
&=&\sum_{C\in \mathfrak{C}_{G-e}(v)\atop u\notin V(C)}(-1)^{|V(C)|}\Phi_{k-l}(H_{V(C)}(G-e))-\beta \sum_{C\in \mathfrak{C}_{G-e}(v)\atop u\notin V(C)}\gamma^{|V(C)|}(-1)^{|V(C)|}\Phi_{k-l}(H_{V(C)\cup\{u\}}(G))\\
&&-\beta \sum_{C\in \mathfrak{C}_{G-e}(v)\atop u\notin V(C)}\gamma^{|V(C)|}(-1)^{|V(C)|}\Phi_{k-l-1}(H_{V(C)\cup\{u\}}(G))\nonumber
\end{eqnarray}
and
\begin{eqnarray}\label{equ3.7}
&&\sum_{C\in \mathfrak{C}_G(v)\atop u\notin V(C)} [\gamma^{|V(C)|}\Phi_{k}(H_{V(C)}(G))\nonumber\\
&=&\sum_{C\in \mathfrak{C}_{G-e}(v)\atop u\notin V(C)}\Phi_{k}\gamma^{|V(C)|}(H_{V(C)}(G-e))-\beta \sum_{C\in \mathfrak{C}_{G-e}(v)\atop u\notin V(C)}\gamma^{|V(C)|}\Phi_{k}(H_{V(C)\cup\{u\}}(G))\\
&&-\beta \sum_{C\in \mathfrak{C}_{G-e}(v)\atop u\notin V(C)}\gamma^{|V(C)|}\Phi_{k-1}(H_{V(C)\cup\{u\}}(G)).\nonumber
\end{eqnarray}

Substituting (\ref{equ3.3})--(\ref{equ3.7}) into (\ref{equ3.2}), we obtain
\begin{eqnarray}\label{equ3.8}
&&\Phi_k(H(G),x)-\Phi_k(H(G-e),x)\nonumber\\
&=& -\beta \Phi_{k}(H_v(G-e))-\beta(x-\beta d(v))[\Phi_{k}(H_{uv}(G)+\Phi_{k-1}(H_{uv}(G)]\nonumber\\
&&-\beta \Phi_{k-1}(H_v(G-e))-\beta(x-\beta d(v))[\Phi_{k-1}(H_{uv}(G)+\Phi_{k-2}(H_{uv}(G)]\nonumber\\
&&+\gamma^{2}[\Phi_{k-2}(H_{uv}(G)+\Phi_{k}(H_{uv}(G)]\nonumber\\
&&-\beta\gamma^{2} \sum_{vw\in E(G)\atop w\neq u}  \Phi_{k-2}(H_{\{v,w,u\}}(G))-\beta\gamma^{2} \sum_{vw\in E(G)\atop w\neq u}  \Phi_{k-3}(H_{\{v,w,u\}}(G))\nonumber\\
&&+\beta\gamma^{2}  \sum_{vw\in E(G)\atop w\neq u}  \Phi_{k}(H_{\{v,w,u\}}(G))+\beta\gamma^{2} \sum_{vw\in E(G)\atop w\neq u}  \Phi_{k-1}(H_{\{v,w,u\}}(G))\nonumber\\
&& +2 \sum_{C\in \mathfrak{C}_G(e))} [(-\gamma)^{|V(C)|}\Phi_{k-l}(H_{V(C)}(G)) - (\gamma)^{|V(C)|}\Phi_k(H_{V(C)}(G))]\nonumber\\
&&-\beta \sum_{C\in \mathfrak{C}_{G-e}(v)\atop u\notin V(C)}\gamma^{|V(C)|}\Phi_{k-l}(H_{V(C)\cup\{u\}}(G))-\beta \sum_{C\in \mathfrak{C}_{G-e}(v)\atop u\notin V(C)}\gamma^{|V(C)|}\Phi_{k-l-1}(H_{V(C)\cup\{u\}}(G)).\nonumber\\
&&-\beta \sum_{C\in \mathfrak{C}_{G-e}(v)\atop u\notin V(C)}\gamma^{|V(C)|}\Phi_{k}(H_{V(C)\cup\{u\}}(G))-\beta \sum_{C\in \mathfrak{C}_{G-e}(v)\atop u\notin V(C)}\gamma^{|V(C)|}\Phi_{k-1}(H_{V(C)\cup\{u\}}(G)).\\
&=& -\beta \Phi_{k}(H_v(G-e))-\beta \Phi_{k-1}(H_v(G-e))+\gamma^{2}[\Phi_{k-2}(H_{uv}(G)+\Phi_{k}(H_{uv}(G)]\nonumber\\
&& +2 \sum_{C\in \mathfrak{C}_G(e))} [(-\gamma)^{|V(C)|}\Phi_{k-1}(H_{V(C)}(G)) - (\gamma)^{|V(C)|}\Phi_k(H_{V(C)}(G))]\nonumber\\
&&-\Bigg\{\beta(x-\beta d(v))[\Phi_{k-1}(H_{uv}(G)+\Phi_{k}(H_{uv}(G)]\nonumber\\
&&+\beta\gamma^{2} \sum_{vw\in E(G)\atop w\neq u}  \Phi_{k-2}(H_{\{v,w,u\}}(G))+\beta\gamma^{2}  \sum_{vw\in E(G)\atop w\neq u}  \Phi_{k}(H_{\{v,w,u\}}(G))\nonumber\\
&&+\beta \sum_{C\in \mathfrak{C}_{G-e}(v)\atop u\notin V(C)}\gamma^{|V(C)|}\Phi_{k-l}(H_{V(C)\cup\{u\}}(G))-\beta \sum_{C\in \mathfrak{C}_{G-e}(v)\atop u\notin V(C)}\gamma^{|V(C)|}\Phi_{k}(H_{V(C)\cup\{u\}}(G))\Bigg\}\nonumber\\
&&-\Bigg\{\beta(x-\beta d(v))[\Phi_{k-2}(H_{uv}(G)+\Phi_{k-1}(H_{uv}(G)]\nonumber\\
&&+\beta\gamma^{2} \sum_{vw\in E(G)\atop w\neq u}  \Phi_{k-3}(H_{\{v,w,u\}}(G))+\beta\gamma^{2}  \sum_{vw\in E(G)\atop w\neq u}  \Phi_{k-1}(H_{\{v,w,u\}}(G))\nonumber\\
&&+\beta \sum_{C\in \mathfrak{C}_{G-e}(v)\atop u\notin V(C)}\gamma^{|V(C)|}\Phi_{k-l-1}(H_{V(C)\cup\{u\}}(G))-\beta \sum_{C\in \mathfrak{C}_{G-e}(v)\atop u\notin V(C)}\gamma^{|V(C)|}\Phi_{k-1}(H_{V(C)\cup\{u\}}(G))\Bigg\}.\nonumber
\end{eqnarray}
To complete the proof, it suffices to establish the following claim.
\begin{claim}\label{cla3.1}
If $R$ is a proper subgraph of $G$ and $v \notin V(R)$, then
\begin{eqnarray*}
&&\Phi_k(H_{V(R)}(G),x)\\
&=&(x-\beta d(v))[\Phi_{k-1}(H_{V(R)\cup\{v\}}(G))+\Phi_k(H_{V(R)\cup\{v\}}(G))] \\
&&+ \gamma^2 \sum_{\mathfrak{u \in N(v) \atop u \notin V(H)}} [\Phi_{k-2}(H_{V(R)\cup\{u,v\}}(G))-\Phi_k(H_{V(R)\cup\{u,v\}}(G))] \\
&&+2 \sum_{C \in \mathfrak{C}_G(v) \atop V(C)\cap V(H)=\emptyset} \gamma^{|V(C)|} [(-1)^{|V(C)|} \Phi_{k-|V(C)|} (H_{V(R)\cup V(C)}(G)) \\
&&-\Phi_{k}(H_{V(R)\cup V(C)}(G))].
\end{eqnarray*}
\end{claim}

 \textbf{Proof of  Claim \ref{cla3.1}.}
The proof is similar to that  of $(i)$ in Theorem \ref{thm1.1}. Without loss of generality, assume $V(R)=\{v_{r+1}, v_{r+2}, \ldots, v_n\}$. Then
\begin{eqnarray}\label{equ3.1}
&&d_k(xI_n - H_{V(R)}(G))\nonumber\\
&=& \sum_{\sigma} \chi_{(k,1^{n-k})}(\sigma)a_{1\sigma(1)}a_{2\sigma(2)}\cdots a_{r\sigma(r)} \nonumber\\
&=& \sum_{\sigma \in S_1} \chi_{(k,1^{n-k})}(\sigma)a_{1\sigma(1)}a_{2\sigma(2)}\cdots a_{r\sigma(r)} \nonumber\\
&& + \sum_{j} \sum_{\sigma \in S_2^{j}} \chi_{(k,1^{n-k})}(\sigma)a_{1\sigma(1)}a_{2\sigma(2)}\cdots a_{r\sigma(r)}+ \sum_{C} \sum_{\sigma \in S_C} \chi_{(k,1^{n-k})}(\sigma)a_{1\sigma(1)}a_{2\sigma(2)}\cdots a_{r\sigma(r)}\nonumber \\
&=& (x-\beta d(v)) \sum_{\sigma'} [\chi_{(k-1,1^{n-k})}(\sigma') + \chi_{(k,1^{n-k-1})}(\sigma')]a_{2\sigma'(2)}\cdots a_{r\sigma'(r)} \nonumber\\
&& +\gamma^2 \sum_{vv_j \in E(G)} \sum_{\sigma'} [\chi_{(k-2,1^{n-k})}(\sigma') - \chi_{(k,1^{n-k-2})}(\sigma')]a_{2\sigma'(2)}\cdots a_{(j-1)\sigma'(j-1)}\nonumber \\
&&\cdot a_{(j+1)\sigma'(j+1)}\cdots a_{r\sigma'(r)} \nonumber\\
&& +2 \sum_{C \in \mathfrak{C}_G(v)} (- \gamma)^{|V(C)|} \sum_{\sigma'} [\chi_{(k-1,1^{n-k})}(\sigma') + (-1)^{|V(C)|-1}\chi_{(k,1^{n-k-l})}(\sigma')]a_{i_1\sigma'(i_1)} \\
&& \cdot a_{i_2\sigma'(i_2)}\cdots a_{i_{n-r-l}\sigma'(i_{n-r-l})} \nonumber\\
&=& (x-\beta d(v)) \sum_{\sigma'} [\chi_{(k-1,1^{n-k})}(\sigma')a_{2\sigma'(2)}\cdots a_{r\sigma'(r)}\nonumber \\
&& +\gamma^2 \sum_{vv_j \in E(G)} \sum_{\sigma'} [\chi_{(k-2,1^{n-k})}(\sigma')a_{2\sigma'(2)}\cdots a_{(j-1)\sigma'(j-1)}a_{(j+1)\sigma'(j+1)}\cdots a_{r\sigma'(r)}\nonumber \\
&& -\chi_{(k,1^{n-k-2})}(\sigma')a_{2\sigma'(2)}\cdots a_{(j-1)\sigma'(j-1)}a_{(j+1)\sigma'(j+1)}\cdots a_{r\sigma'(r)}] \nonumber\\
&& +2 \sum_{C \in \mathfrak{C}_G(v)} (- \gamma)^{|V(C)|} \sum_{\sigma'} [\chi_{(k-1,1^{n-k})}(\sigma')a_{i_1\sigma'(i_1)}a_{i_2\sigma'(i_2)}\cdots a_{i_{n-r-l}\sigma'(i_{n-r-l})} \nonumber\\
&& +(-1)^{|V(C)|-1}\chi_{(k,1^{n-k-l})}(\sigma')a_{i_1\sigma'(i_1)}a_{i_2\sigma'(i_2)}\cdots a_{i_{n-r-l}\sigma'(i_{n-r-l})}] \nonumber\\
&=& (x-\beta d(v))[\Phi_{k-1}(H_{V(R)\cup\{v\}}(G)) + \Phi_{k}(H_{V(R)\cup\{v\}}(G))] \nonumber\\
&& +\gamma^2 \sum_{vv_j \in E(G)} [\Phi_{k-2}(H_{V(R)\cup\{v_j,v\}}(G)) - \Phi_{k}(H_{V(R)\cup\{v_j,v\}}(G))] \nonumber\\
&& +2 \sum_{C \in \mathfrak{C}_G(v)} \gamma^{|V(C)|}[(-1)^{|V(C)|} \Phi_{k-|V(C)|}(H_{V(R)\cup V(C)}(G)) - \Phi_{k}(H_{V(R)\cup V(C)}(G))],\nonumber
\end{eqnarray}
where in Equation (\ref{equ3.1}), $i_1, \ldots, i_{n-l}$ denote indices of the vertices in the graph $G-V(C)$.

By Claim \ref{cla3.1}, we have
\begin{eqnarray}\label{equ3.9}
&&\Phi_k(H_{u}(G-e))\nonumber \\
&=& (x-\beta (d(v)-1))[\Phi_{k-1}(H_{uv}(G-e)) + \Phi_k(H_{uv}(G-e))]\nonumber \\
&& + \gamma^2 \sum_{vw\in E(G-e)\atop w\neq u} [\Phi_{k-2}(H_{\{v,w,u\}}(G-e)) - \Phi_k(H_{\{v,w,u\}}(G-e))] \\
&& +2 \sum_{C\in \mathfrak{C}_{G-e}(v)\atop u\notin V(C)} \gamma^{|V(C)|}[(-1)^{|V(C)|}\Phi_{k-l}(H_{V(C)\cup\{u\}}(G)) - \Phi_k(H_{V(C)\cup\{u\}}(G))].\nonumber
\end{eqnarray}
and
\begin{eqnarray}\label{equ3.10}
&&\Phi_{k-1}(H_{u}(G-e))\nonumber \\
&=& (x-\beta (d(v)-1))[\Phi_{k-2}(H_{uv}(G-e)) + \Phi_{k-1}(H_{uv}(G-e))]\nonumber \\
&& + \gamma^2 \sum_{vw\in E(G-e)\atop w\neq u} [\Phi_{k-3}(H_{\{v,w,u\}}(G-e)) - \Phi_{k-1}(H_{\{v,w,u\}}(G-e))] \\
&& +2 \sum_{C\in \mathfrak{C}_{G-e}(v)\atop u\notin V(C)} \gamma^{|V(C)|}[(-1)^{|V(C)|}\Phi_{k-l-1}(H_{V(C)\cup\{u\}}(G)) - \Phi_k(H_{V(C)\cup\{u\}}(G))].\nonumber
\end{eqnarray}

Combining (\ref{equ3.8}), (\ref{equ3.9}) and (\ref{equ3.10}), we get
\begin{eqnarray*}
\Phi_k(H(G),x)
&=& \Phi_k(H(G-e),x) - \beta[\Phi_k(H_v(G-e)) + \Phi_{k-1}(H_v(G-e))] \\
&& - \beta[\Phi_k(H_u(G-e)) + \Phi_{k-1}(H_u(G-e))] \\
&& + [(\beta^2 - \gamma^2)\Phi_k(H_{uv}(G)) + 2\beta^2 \Phi_{k-1}(H_{uv}(G)) + (\beta^2 + \gamma^2)\Phi_{k-2}(H_{uv}(G))] \\
&& +2 \sum_{C\in \mathfrak{C}_G(e))} \gamma^{|V(C)|}[(-1)^{|V(C)|}\Phi_{k-|V(C)|}(H_{V(C)}(G)) - \Phi_k(H_{V(C)}(G))].
\end{eqnarray*}
This completes the proof of  Theorem \ref{thm1.1}. \qed

\section{Proof of Theorem \ref{thm1.3}}
In this section, we will prove Theorem \ref{thm1.3}.

\textbf{Proof of Theorem \ref{thm1.3}.} $(i)$
Suppose that the vertices of $\overrightarrow{G}$ are $\{v_1, v_2, \ldots, v_n\}$. Index the rows and columns of $H(\overrightarrow{G})$ by $v_1, v_2, \ldots, v_n$. Consider a term $a_{1\sigma(1)}a_{2\sigma(2)}\cdots a_{n\sigma(n)}$ in the expansion of $d_k(xI_n - \beta D(\overrightarrow{G}) - \gamma A(\overrightarrow{G}))$, where $a_{ij}$ is the $(i,j)$-th entry of $xI_n - \beta D(\overrightarrow{G}) - \gamma A(\overrightarrow{G})$. It is easy to see that $a_{ii} = x - \beta d^{+}(v_i)$, and if $(v_i, v_j) \in E(\overrightarrow{G})$, then $a_{ij} = -\gamma$; otherwise, it is $0$. Therefore, if $a_{1\sigma(1)}a_{2\sigma(2)}\cdots a_{n\sigma(n)} \neq 0$, then either $j = \sigma(j)$ or $(v_j, v_{\sigma(j)}) \in E(\overrightarrow{G})$. In the following, we only consider permutations $\sigma$ for which $a_{1\sigma(1)}a_{2\sigma(2)}\cdots a_{n\sigma(n)} \neq 0$. Define $\mathcal{P}= \{\sigma \mid a_{1\sigma(1)}a_{2\sigma(2)}\cdots a_{n\sigma(n)} \neq 0\}$. Without loss of generality, let $a_{11} = x - \beta d^{+}(v)$ correspond to the vertex $v$. As each $\sigma$ can be expressed as a product of disjoint cycles, then we write $\sigma = \gamma_1 \sigma'$, where $\gamma_1$ is a cycle. Depending on the length of $\gamma_1$, we classify $\sigma$  into one of the following cases:
\begin{eqnarray*}
&&S_{1}=\{\sigma\in \mathcal{P} \mid \gamma_{1}  \text{ corresponds to } v 
\text{ not being in any orientation of}~G~ \text{with }\\
&&\text{a length greater than } 0\}, \\
&&S_{\overrightarrow{C}}=\{\sigma\in \mathcal{P} \mid \gamma_{1} \text{ corresponds to a consistent directed cycle } \overrightarrow{C} \text{ in } \overrightarrow{G} \\
&&\text{ containing the vertex } v \text{ with length } l \geq 3\}.
\end{eqnarray*}
By Theorem \ref{lem2.1}, the following holds:\\
$(a)$ For $\sigma \in S_1$, we have
$$\chi_{(k,1^{n-k})}(\sigma) = \chi_{(k-1,1^{n-k})}(\sigma') + \chi_{(k,1^{n-k-1})}(\sigma'),$$
$(b)$  For $\sigma \in S_{\overrightarrow{C}}$, we have
$$\chi_{(k,1^{n-k})}(\sigma) = \chi_{(k-1,1^{n-k})}(\sigma') + (-1)^{l-1}\chi_{(k,1^{n-k-1})}(\sigma').$$
Therefore,
\begin{eqnarray}\label{equ3.1.1}
&&d_k(xI_n - \beta D(\overrightarrow{G}) - \gamma A(\overrightarrow{G})) \nonumber\\
&=& \sum_{\sigma} \chi_{(k,1^{n-k})}(\sigma)a_{1\sigma(1)}a_{2\sigma(2)}\cdots a_{n\sigma(n)}\nonumber \\
&=& \sum_{\sigma \in S_1} \chi_{(k,1^{n-k})}(\sigma)a_{1\sigma(1)}a_{2\sigma(2)}\cdots a_{n\sigma(n)} + \sum_{\overrightarrow{C}} \sum_{\sigma \in S_{\overrightarrow{C}}} \chi_{(k,1^{n-k})}(\sigma)a_{1\sigma(1)}a_{2\sigma(2)}\cdots a_{n\sigma(n)} \nonumber\\
&=& (x-\beta d^{+}(v)) \sum_{\sigma' \in S_1} [\chi_{(k-1,1^{n-k})}(\sigma')+\chi_{(k,1^{n-k-1})}(\sigma')]a_{2\sigma'(2)} \cdots a_{n\sigma'(n)}\nonumber \\
&& + \sum_{\overrightarrow{C} \in \mathcal{C}_{\overrightarrow{G}}(v)} (- \gamma)^{|V(\overrightarrow{C})|} \sum_{\sigma' \in S_{\overrightarrow{C}}} [\chi_{(k-l,1^{n-k})}(\sigma')+(-1)^{l-1}\chi_{(k,1^{n-k-l})}(\sigma')]a_{i_1\sigma^{'}(i_1)} \nonumber\\
&&\cdot a_{i_2\sigma^{'}(i_2)} \cdots a_{i_{n-l}\sigma^{'}(i_{n-l})} \\
&=& (x-\beta d^{+}(v)) \sum_{\sigma' \in S_1} [\chi_{(k-1,1^{n-k})}(\sigma')a_{2\sigma'(2)} \cdots a_{n\sigma'(n)} + \chi_{(k,1^{n-k-1})}(\sigma')a_{2\sigma'(2)} \cdots a_{n\sigma'(n)}]\nonumber \\
&& + \sum_{\overrightarrow{C} \in \mathcal{C}_{\overrightarrow{G}}(v)} (- \gamma)^{|V(\overrightarrow{C})|} \sum_{\sigma' \in S_{\overrightarrow{C}}} [\chi_{(k-l,1^{n-k})}(\sigma')a_{i_1\sigma^{'}(i_1)}a_{i_2\sigma^{'}(i_2)} \cdots a_{i_{n-l}\sigma^{'}(i_{n-l})} \nonumber\\
&& + (-1)^{l-1}\chi_{(k,1^{n-k-l})}(\sigma')a_{i_1\sigma^{'}(i_1)}a_{i_2\sigma^{'}(i_2)} \cdots a_{i_{n-l}\sigma^{'}(i_{n-l})}] \nonumber\\
&=& (x-\beta d^{+}(v))[\Phi_{k-1}(H_v(\overrightarrow{G}))+\Phi_k(H_v(\overrightarrow{G}))]\nonumber\\
&& + \sum_{\overrightarrow{C} \in \mathcal{C}_{\overrightarrow{G}}(v)} \gamma^{|V(\overrightarrow{C})|}[(-1)^{|V(\overrightarrow{C})|}\Phi_{k-|V(\overrightarrow{C})|}(H_{V(\overrightarrow{C})}(\overrightarrow{G}))
-\Phi_k(H_{V(\overrightarrow{C})}(\overrightarrow{G}))],\nonumber
\end{eqnarray}
where in Equation (\ref{equ3.1.1}), $i_1,\ldots,i_{n-l}$ denote the labels of the vertices in $\overrightarrow{G}-V(\overrightarrow{C})$.

$(ii)$ Applying Theorem \ref{thm1.3} $(i)$ to $\overrightarrow{G}$ and $\overrightarrow{G}-\overrightarrow{e}$ respectively, we have derived that
\begin{eqnarray*}
&&\Phi_k(H(\overrightarrow{G}),x)\\
&=& (x-\beta d^{+}(v))[\Phi_{k-1}(H_v(\overrightarrow{G}))+\Phi_k(H_v(\overrightarrow{G}))]\nonumber\\
&& + \sum_{\overrightarrow{C} \in \mathcal{C}_{\overrightarrow{G}}(v)} \gamma^{|V(\overrightarrow{C})|}[(-1)^{|V(\overrightarrow{C})|}\Phi_{k-|V(\overrightarrow{C})|}(H_{V(\overrightarrow{C})}(\overrightarrow{G}))
-\Phi_k(H_{V(\overrightarrow{C})}(\overrightarrow{G}))]
\end{eqnarray*}
and
\begin{eqnarray*}
&&\Phi_k(H(\overrightarrow{G}-\overrightarrow{e}),x)\\
&=& (x-\beta (d^{+}(v)-1))[\Phi_{k-1}(H_v(\overrightarrow{G}-\overrightarrow{e}))+\Phi_k(H_v(\overrightarrow{G}-\overrightarrow{e}))]\nonumber\\
&& + \sum_{\overrightarrow{C} \in \mathcal{C}_{\overrightarrow{G}-\overrightarrow{e}}(v)} \gamma^{|V(\overrightarrow{C})|}[(-1)^{|V(\overrightarrow{C})|}\Phi_{k-|V(\overrightarrow{C})|}(H_{V(\overrightarrow{C})}(\overrightarrow{G}-\overrightarrow{e}))
-\Phi_k(H_{V(\overrightarrow{C})}(\overrightarrow{G}-\overrightarrow{e}))].
\end{eqnarray*}
Thus,
\begin{eqnarray}\label{equa4.2}
&&\Phi_k(H(\overrightarrow{G}),x)-\Phi_k(H(\overrightarrow{G}-\overrightarrow{e}),x)\nonumber\\
&=& \Bigg\{(x-\beta d^{+}(v))[\Phi_{k-1}(H_v(\overrightarrow{G}))+\Phi_k(H_v(\overrightarrow{G}))]\nonumber\\
&&-(x-\beta (d^{+}(v)-1))[\Phi_{k-1}(H_v(\overrightarrow{G}-\overrightarrow{e}))+\Phi_k(H_v(\overrightarrow{G}-\overrightarrow{e}))]\Bigg\} \\
&& + \Bigg\{\sum_{\overrightarrow{C} \in \mathcal{C}_{\overrightarrow{G}}(v)} \gamma^{|V(\overrightarrow{C})|}[(-1)^{|V(\overrightarrow{C})|}\Phi_{k-|V(\overrightarrow{C})|}(H_{V(\overrightarrow{C})}(\overrightarrow{G}))
-\Phi_k(H_{V(\overrightarrow{C})}(\overrightarrow{G}))]\nonumber\\
&&- \sum_{\overrightarrow{C} \in \mathcal{C}_{\overrightarrow{G}-\overrightarrow{e}}(v)} \gamma^{|V(\overrightarrow{C})|}[(-1)^{|V(\overrightarrow{C})|}\Phi_{k-|V(\overrightarrow{C})|}(H_{V(\overrightarrow{C})}
(\overrightarrow{G}-\overrightarrow{e}))
-\Phi_k(H_{V(\overrightarrow{C})}(\overrightarrow{G}-\overrightarrow{e}))]\Bigg\}.\nonumber
\end{eqnarray}
Observing the structure of the matrices $H_v(\overrightarrow{G})$ and $H_v(\overrightarrow{G}-\overrightarrow{e})$, we obtain that
\begin{eqnarray}\label{equa4.3}
H_v(\overrightarrow{G})=H_v(\overrightarrow{G}-\overrightarrow{e}).
\end{eqnarray}
Similarly, we have
\begin{eqnarray}\label{equa4.4}
\sum\limits_{\overrightarrow{C} \in \mathcal{C}_{\overrightarrow{G}}(v)\atop u\notin V(\overrightarrow{C})} \gamma^{|V(\overrightarrow{C})|}H_{V(\overrightarrow{C})}
(\overrightarrow{G}))=\sum\limits_{\overrightarrow{C} \in \mathcal{C}_{\overrightarrow{G}-\overrightarrow{e}}(v)} \gamma^{|V(\overrightarrow{C})|}H_{V(\overrightarrow{C})}(\overrightarrow{G}-\overrightarrow{e}).
\end{eqnarray}
Combining (\ref{equa4.2}), (\ref{equa4.3}) and (\ref{equa4.4}), we conclude that
\begin{eqnarray*}
\Phi_k(H(\overrightarrow{G}),x)
&=& \Phi_k(H(\overrightarrow{G}-\overrightarrow{e}),x) - \beta[\Phi_k(H_v(\overrightarrow{G}-\overrightarrow{e})) + \Phi_{k-1}(H_v(\overrightarrow{G}-\overrightarrow{e}))] \\
&& + \sum_{\overrightarrow{C}\in \mathcal{C}_{\overrightarrow{G}}(e)} (\gamma)^{|V(\overrightarrow{C})|}[(-1)^{|V(\overrightarrow{C})|}\Phi_{k-|V(\overrightarrow{C})|}(H_{V(\overrightarrow{C})}(\overrightarrow{G})) - \Phi_k(H_{V(\overrightarrow{C})}(\overrightarrow{G}))].
\end{eqnarray*}

This completes the proof of Theorem \ref{thm1.3}. \qed

\section{Applications of Theorems \ref{thm1.1} and \ref{thm1.2}}
\subsection{Applications of Theorem \ref{thm1.1} }
By leveraging the recursive formulas for the hook immanantal polynomials of $H(G)$, we can derive the corresponding recursive formulas for several key graph matrices, including $L(G)$, $Q(G)$, $A(G)$ and $A_{\alpha}(G)$.

First, we present the recursive formula for the  hook immanantal polynomial of the Laplacian matrix.
\begin{corollary}\label{coro4.1}
$(i)$ Let $v$ be a vertex of graph $G$. Then
\begin{eqnarray*}
\Phi_k(L(G),x)
&=& (x- d(v))[\Phi_{k-1}(L_v(G)) + \Phi_k(L_v(G))] \\
&& + \sum_{u\in N(v)} [\Phi_{k-2}(L_{uv}(G)) - \Phi_k(L_{uv}(G))] \\
&& +2 \sum_{C\in \mathfrak{C}_G(v)}[\Phi_{k-|V(C)|}(L_{V(C)}(G)) - (-1)^{|V(C)|}\Phi_k(L_{V(C)}(G))].
\end{eqnarray*}

$(ii)$ Let $e=uv$ be an edge of graph $G$. Then
\begin{eqnarray*}
\Phi_k(L(G),x)
&=& \Phi_k(L(G-e),x) - [\Phi_k(L_v(G-e)) + \Phi_{k-1}(L_v(G-e))] \\
&& - [\Phi_k(L_u(G-e)) + \Phi_{k-1}(L_u(G-e))] \\
&& + 2[\Phi_{k-1}(L_{uv}(G)) + \Phi_{k-2}(L_{uv}(G))] \\
&& +2 \sum_{C\in \mathfrak{C}_G(e))}[\Phi_{k-|V(C)|}(L_{V(C)}(G)) - (-1)^{|V(C)|}\Phi_k(L_{V(C)}(G))].
\end{eqnarray*}
\end{corollary}
Corollary \ref{coro4.1} directly yields the recursive formulas for Laplacian characteristic polynomials of graphs that were introduced by Guo et al. \cite{guo}.
\begin{corollary}(Guo et al., \cite{guo})\label{corol4.1}
$(i)$  Let $v$ be a vertex of $G$. Then
\begin{eqnarray*}
\varphi(L(G)) = (x - d(v))\varphi(L_v(G)) - \sum_{u \in N(v)} \varphi(L_{uv}(G)) - 2 \sum_{C \in \mathfrak{C}_G(v)} (-1)^{|V(C)|}\varphi(L_{V(C)}(G)).
\end{eqnarray*}

$(ii)$ Let $e = uv$ be an edge of $G$. Then
\begin{eqnarray*}
\varphi(L(G)) = \varphi(L(G - e)) - \varphi(L_u(G - e)) - \varphi(L_v(G - e)) - 2 \sum_{C \in \mathfrak{C}_G(e)} (-1)^{|V(C)|}\varphi(L_{V(C)}(G)).
\end{eqnarray*}
\end{corollary}
Using Corollary  \ref{coro4.1}, we can derive the recursive formula for the second immanantal polynomial of a graph, which was proposed by Wu et al. \cite{wu2}.
\begin{corollary}(Wu et al., \cite{wu2})
$(i)$ Let $v$ be a vertex of graph $G$. Then
    \begin{eqnarray*}
    \Phi_2(L(G),x)
    &=& \sum_{i=1}^{n-1} (x - d(v_i))\varphi(L_{v_i}(G)) + \sum_{w} \varphi(L_{vw}(G)) \\
    &&+ 2 \sum_{C \in \mathfrak{C}_G(v)} (-1)^{|V(C)|}\varphi(L_{V(C)}(G)).
    \end{eqnarray*}

$(ii)$ Let $G$ be a graph with $n$ vertices. Suppose that $e = u_j v_j$ is an edge of $G$ and  $u_i \in V(G)-\{u_j, v_j\}$. Then
    \begin{eqnarray*}
    &&\Phi_2(L(G),x)\\
     &=& \Phi_2(L(G-e),x)+ \left[ \sum_{i=1}^{n-2} (x - d_{u_i})\varphi(L_{u_i}(G)) - \sum_{i=1}^{n-2}(x - d_{u_i} + 1)\varphi(L_{u_i}(G - e)) \right]+ \varphi(L_{v_j}(G - e)) \\
    &&- \varphi(L_{u_j}(G - e)) - 2(x - d(u_j) + 1)\varphi(L_{u_j v_j}(G))+ 2 \sum_{C \in \mathfrak{C}_G(e)} (-1)^{|V(C)|}\varphi(L_{V(C)}(G)).
    \end{eqnarray*}
\end{corollary}

The recursive formulas to calculate the  Laplacian  permanental polynomials of graphs were given by Liu and Wu \cite{liu}, with this result following directly from Corollary \ref{coro4.1}.
\begin{corollary}(Liu and Wu, \cite{liu})
$(i)$ Let $v$ be a vertex of graph $G$. Then
\begin{eqnarray*}
\psi(L(G),x)
&=& (x- d(v))\psi(L_v(G))+ \sum_{u\in N(v)} \psi(L_{uv}(G))  +2 \sum_{C\in \mathfrak{C}_G(v)}[\psi(L_{V(C)}(G)).
\end{eqnarray*}

$(ii)$ Let $e=uv$ be an edge of graph $G$. Then
\begin{eqnarray*}
\psi(L(G),x)
&=& \psi(L(G-e),x) - \psi(L_v(G-e)) -  \psi(L_u(G-e))\\
&&+ 2 \psi(L_{uv}(G))+2 \sum_{C\in \mathfrak{C}_G(e))}\psi(L_{V(C)}(G)).
\end{eqnarray*}
\end{corollary}
Subsequently, we derive the recursive formula for the hook immanantal polynomial of the signless Laplacian matrix.
\begin{corollary}\label{coro4.2}
$(i)$ Let $v$ be a vertex of graph $G$. Then
\begin{eqnarray*}
\Phi_k(Q(G),x)
&=& (x- d(v))[\Phi_{k-1}(Q_v(G)) + \Phi_k(Q_v(G))] \\
&& + \sum_{u\in N(v)} [\Phi_{k-2}(Q_{uv}(G)) - \Phi_k(Q_{uv}(G))] \\
&& +2 \sum_{C\in \mathfrak{C}_G(v)}[(-1)^{|V(C)|}\Phi_{k-|V(C)|}(Q_{V(C)}(G)) - \Phi_k(Q_{V(C)}(G))].
\end{eqnarray*}

$(ii)$ Let $e=uv$ be an edge of graph $G$. Then
\begin{eqnarray*}
\Phi_k(Q(G),x)
&=& \Phi_k(Q(G-e),x) - [\Phi_k(Q_v(G-e)) + \Phi_{k-1}(Q_v(G-e))] \\
&& - [\Phi_k(Q_u(G-e)) + \Phi_{k-1}(Q_u(G-e))] \\
&& + 2[\Phi_{k-1}(Q_{uv}(G)) + \Phi_{k-2}(Q_{uv}(G))] \\
&& +2 \sum_{C\in \mathfrak{C}_G(e))}[(-1)^{|V(C)|}\Phi_{k-|V(C)|}(Q_{V(C)}(G)) - \Phi_k(Q_{V(C)}(G))].
\end{eqnarray*}
\end{corollary}
Corollary \ref{coro4.2}  immediately yields the recursive formulas for the signless Laplacian characteristic polynomials of graphs that were introduced by Guo et al. \cite{guo}.
\begin{corollary}(Guo et al., \cite{guo})
$(i)$  Let $v$ be a vertex of $G$. Then
\begin{eqnarray*}
\varphi(Q(G),x) = (x - d(v))\varphi(Q_v(G)) - \sum_{u \in N(v)} \varphi(Q_{uv}(G)) - 2 \sum_{C \in \mathfrak{C}_G(v)} \varphi(Q_{V(C)}(G)).
\end{eqnarray*}
$(ii)$ Let $e = uv$ be an edge of $G$. Then
\begin{eqnarray*}
\varphi(Q(G),x) = \varphi(QL(G - e),x) - \varphi(Q_u(G - e)) - \varphi(Q_v(G - e)) - 2 \sum_{C \in \mathfrak{C}_G(e)} \varphi(Q_{V(C)}(G)).
\end{eqnarray*}
\end{corollary}

Liu and Wu \cite{liu} proposed recursive formulas for computing the signless Laplacian permanental polynomials of graphs. This result can be immediately derived from Corollary \ref{coro4.2}.

\begin{corollary}(Liu and Wu, \cite{liu})
$(i)$ Let $v$ be a vertex of graph $G$. Then
\begin{eqnarray*}
\psi(Q(G),x)
&=& (x- d(v))\psi(Q_v(G))+ \sum_{u\in N(v)} \psi(Q_{uv}(G))  +2 \sum_{C\in \mathfrak{C}_G(v)}(-1)^{|V(C)|}\psi(Q_{V(C)}(G)).
\end{eqnarray*}

$(ii)$ Let $e=uv$ be an edge of graph $G$. Then
\begin{eqnarray*}
\psi(Q(G),x)
&=& \psi(Q(G-e),x) - \psi(Q_v(G-e)) -  \psi(Q_u(G-e))\\
&&+ 2 \psi(Q_{uv}(G))+2 \sum_{C\in \mathfrak{C}_G(e))}(-1)^{|V(C)|}\psi(Q_{V(C)}(G)).
\end{eqnarray*}
\end{corollary}
The following result was first presented in \cite{yu}. In this paper, we provide an alternative proof by  using Theorem \ref{thm1.1}.
\begin{corollary}\label{coro4.5}(Yu and Qu, \cite{yu})
Let $G$ be a bipartite graph with $n$ vertices. Then
$$\Phi_k(L(G), x) = \Phi_k(Q(G), x).$$
\end{corollary}

\begin{proof}
It is easy to verify that for any vertex $v$ of $G$,
$\Phi_k(L_{V(G-v)}(G), x) = \Phi_k(Q_{V(G-v)}(G), x)$, and for any two vertices $u$ and $v$ of $G$,
$\Phi_k(L_{V(G-\{u,v\})}(G), x) = \Phi_k(Q_{V(G-\{u,v\})}(G), x)$.
Since $G$ is bipartite, it contains no odd cycles. By repeatedly applying Claim \ref{cla3.1}, we obtain that for any subset $S \subseteq V(G)$ with $1\leq |S|\leq n-1$,
$\Phi_k(L_{S}(G), x) = \Phi_k(Q_{S}(G), x)$.
Then the required conclusion was obtained by  Theorem \ref{thm1.1}.
\end{proof}
The identity between the permanental polynomials of the Laplacian and signless Laplacian matrices for bipartite graphs, established by Faria \cite{far}, can alternatively be obtained as a direct consequence of Corollaries \ref{coro4.1} and \ref{coro4.2}.
\begin{corollary}(Faria, \cite{far})\label{cor4.13}
Let $G$ be a bipartite graph with $n$ vertices. Then
$$\psi(L(G), x) = \psi(Q(G), x).$$
\end{corollary}

A recurrence relation is derived for evaluating the hook immanantal polynomial of the adjacency matrix.
\begin{corollary}\label{coro4.7}
$(i)$ Let $v$ be a vertex of graph $G$. Then
\begin{eqnarray*}
\Phi_k(A(G),x)
&=& x[\Phi_{k-1}(A(G-v)) + \Phi_k(A(G-v))] \\
&& + \sum_{u\in N(v)} [\Phi_{k-2}(A(G-u-v)) - \Phi_k(A(G-u-v))] \\
&& +2 \sum_{C\in \mathfrak{C}_G(v)} [(-1)^{|V(C)|}\Phi_{k-|V(C)|}(A(G-V(C))) - \Phi_k(A(G-V(C)))].
\end{eqnarray*}

$(ii)$ Let $e=uv$ be an edge of graph $G$. Then
\begin{eqnarray*}
\Phi_k(A(G),x)
&=& \Phi_k(A(G-e),x) + [\Phi_{k-2}(A(G-u-v))-\Phi_k(A(G-u-v))] \\
&& +2 \sum_{C\in \mathfrak{C}_G(e))} [(-1)^{|V(C)|}\Phi_{k-|V(C)|}(A(G-V(C))) - \Phi_k(A(G-V(C)))].
\end{eqnarray*}
\end{corollary}

\begin{proof}
When $\beta=0$ and $\gamma=1$, applying Theorem \ref{thm1.1}, we conclude that
\begin{eqnarray}\label{equ4.1}
\Phi_k(A(G),x)
&=& x[\Phi_{k-1}(A_v(G)) + \Phi_k(A_v(G))]\nonumber \\
&& + \sum_{u\in N(v)} [\Phi_{k-2}(A_{uv}(G)) - \Phi_k(A_{uv}(G))] \\
&& +2 \sum_{C\in \mathfrak{C}_G(v)} [(-1)^{|V(C)|}\Phi_{k-|V(C)|}(A_{V(C)}(G)) - \Phi_k(A_{V(C)}(G))].\nonumber
\end{eqnarray}
\begin{eqnarray}\label{equ4.2}
\Phi_k(A(G),x)
&=& \Phi_k(A(G-e),x) + [\Phi_{k-2}(A_{uv}(G))-\Phi_k(A_{uv}(G))]\nonumber \\
&& +2 \sum_{C\in \mathfrak{C}_G(e))} [(-1)^{|V(C)|}\Phi_{k-|V(C)|}(A_{V(C)}(G)) - \Phi_k(A_{V(C)}(G))].
\end{eqnarray}
By examining the structure of matrices $A_v(G)$ and $A(G-v)$, we have
\begin{eqnarray}\label{equ4.3}
A_v(G)=A(G-v)
\end{eqnarray}
Similarly, we get
\begin{eqnarray}\label{equ4.4}
A_{uv}(G)&=&A(G-u-v)
\end{eqnarray}
\begin{eqnarray}\label{equ4.5}
A_{V(C)}(G)&=&A(G-V(C)).
\end{eqnarray}
Combining (\ref{equ4.1})--(\ref{equ4.5}), we obtain the required results.
\end{proof}
In \cite{cvet}, Cvetkovi\'{c} et al.  determined  the recursive formulas to calculate the  characteristic polynomials of the adjacency matrices of graphs. This results is implied by Corollary \ref{coro4.7}.
\begin{corollary}(Cvetkovi\'{c} et al., \cite{cvet})
$(i)$ Let $v$ be a vertex of graph $G$. Then
\begin{eqnarray*}
\varphi(A(G),x)
&=& x\varphi(A(G-v)) - \sum_{u\in N(v)} \varphi(A(G-u-v))-2 \sum_{C\in \mathfrak{C}_G(v)} \varphi(A(G-V(C))).
\end{eqnarray*}

$(ii)$ Let $e=uv$ be an edge of graph $G$. Then
\begin{eqnarray*}
\varphi(A(G),x)= \varphi(A(G-e)) - \varphi(A(G-u-v))-2 \sum_{C\in \mathfrak{C}_G(e))} \varphi(A(G-V(C))).
\end{eqnarray*}
\end{corollary}
As a direct consequence of Corollary \ref{coro4.7}, Borowiecki and J\'{o}zwia\'{k} \cite{bor} investigated  the recursive formulas to calculate the  permanental polynomials of the adjacency matrices of graphs.
\begin{corollary}(Borowiecki and J\'{o}zwia\'{k}, \cite{bor})
$(i)$ Let $v$ be a vertex of graph $G$. Then
\begin{eqnarray*}
\psi(A(G),x)
&=& x\psi(A(G-v))+ \sum_{u\in N(v)} \psi A(G-u-v))+2 \sum_{C\in \mathfrak{C}_G(v)}(-1)^{|V(C)|} \psi (A(G-V(C))).
\end{eqnarray*}

$(ii)$ Let $e=uv$ be an edge of graph $G$. Then
\begin{eqnarray*}
\psi(A(G),x)= \psi(A(G-e)) + \psi(A(G-u-v))+2 \sum_{C\in \mathfrak{C}_G(e))} (-1)^{|V(C)|}\psi(A(G-V(C))).
\end{eqnarray*}
\end{corollary}

Nikiforov  \cite{nik} proposed the concept of the $A_{\alpha}$ matrix, which is closely related to the resultant properties of graphs. The $A_{\alpha}$ matrix has been extensively studied by many scholars \cite{lin,xue}. We establish a recurrence relation for the hook immanantal polynomial of the $A_{\alpha}$ matrix of graphs.

\begin{corollary}\label{the4.6}
$(i)$ Let $v$ be a vertex of graph $G$. Then
\begin{eqnarray*}
&&\Phi_k(A_{\alpha}(G),x)\\
&=& (x-\alpha d(v))[\Phi_{k-1}((A_{\alpha})_v(G)) + \Phi_k((A_{\alpha})_v(G))] \\
&& + (1-\alpha)^2 \sum_{u\in N(v)} [\Phi_{k-2}((A_{\alpha})_{uv}(G)) - \Phi_k((A_{\alpha})_{uv}(G))] \\
&& +2 \sum_{C\in \mathfrak{C}_G(v)} (1-\alpha)^{|V(C)|}[(-1)^{|V(C)|}\Phi_{k-|V(C)|}((A_{\alpha})_{V(C)}(G)) - \Phi_k((A_{\alpha})_{V(C)}(G))].
\end{eqnarray*}

$(ii)$ Let $e=uv$ be an edge of graph $G$. Then
\begin{eqnarray*}
&&\Phi_k(A_{\alpha}(G),x)\\
&=& \Phi_k(A_{\alpha}(G-e),x) - \alpha[\Phi_k((A_{\alpha})_v(G-e)) + \Phi_{k-1}((A_{\alpha})_v(G-e))] \\
&& - \alpha[\Phi_k((A_{\alpha})_u(G-e)) + \Phi_{k-1}((A_{\alpha})_u(G-e))] \\
&& + [(2\alpha -1)\Phi_k((A_{\alpha})_{uv}(G)) + 2\alpha^2 \Phi_{k-1}((A_{\alpha})_{uv}(G)) + (2\alpha^2 +2\alpha+1)\Phi_{k-2}((A_{\alpha})_{uv}(G))] \\
&& +2 \sum_{C\in \mathfrak{C}_G(e))} (1-\alpha)^{|V(C)|}[(-1)^{|V(C)|}\Phi_{k-|V(C)|}(H(A_{\alpha})_{V(C)}(G)) - \Phi_k((A_{\alpha})_{V(C)}(G))].
\end{eqnarray*}
\end{corollary}

\subsection{Applications of Theorem \ref{thm1.2} }

Using  the recursive formulas for  hook immanants of $H(G)$, we can obtain the recursive formulas for hook immanant  of several important graph matrices such as $L(G)$, $Q(G)$, $A(G)$ and $A_{\alpha}(G)$.

A recursive computational scheme is derived for evaluating the hook immanant of  $L(G)$.
\begin{corollary}\label{coro4.2.1}
$(i)$ Let $v$ be a vertex of graph $G$. Then
\begin{eqnarray*}
d_k(L(G))
&=&  d(v)[d_{k-1}(L_v(G)) + d_k(L_v(G))] + \sum_{u\in N(v)} [d_{k-2}(L_{uv}(G)) - d_k(L_{uv}(G))] \\
&& +2 \sum_{C\in \mathfrak{C}_G(v)}[(-1)^{|V(C)|}d_{k-|V(C)|}(L_{V(C)}(G)) -d_k(L_{V(C)}(G))].
\end{eqnarray*}

$(ii)$ Let $e=uv$ be an edge of graph $G$. Then
\begin{eqnarray*}
d_k(L(G))
&=& d_k(L(G-e)) +[d_k(L_v(G-e)) + d_{k-1}(L_v(G-e))] \\
&& + [d_k(L_u(G-e)) +d_{k-1}(L_u(G-e))]+ 2[d_{k-1}(L_{uv}(G)) + d_{k-2}(L_{uv}(G))] \\
&& +2 \sum_{C\in \mathfrak{C}_G(e))}[(-1)^{|V(C)|}d_{k-|V(C)|}(L_{V(C)}(G)) - d_k(L_{V(C)}(G))].
\end{eqnarray*}
\end{corollary}
Wu et al. \cite{wu1} investigated vertex-deletion and edge-deletion   formulas for the permanents of Laplacian matrices, a fact that can be derived from Corollary \ref{coro4.2.1}.
\begin{corollary}(Wu et al., \cite{wu1})
$(i)$ Let $v$ be a vertex of graph $G$. Then
\begin{eqnarray*}
{\rm per} L(G)=d(v){\rm per} L_{v}(G)+\sum\limits_{u\in N(v)}{\rm per }L_{vu}(G)+2\sum\limits_{C\in \mathfrak {C}_{G}(v)}(-1)^{|V(C)|}{\rm per }L_{V(C)}(G).
\end{eqnarray*}

$(ii)$ Let $e=uv$ be an edge of graph $G$. Then
\begin{eqnarray*}
{\rm per} L(G)&=&{\rm per} L(G-e)+{\rm per} L_{v}(G-e)+{\rm per} L_{u}(G-e)\\
&&+2{\rm per} L_{uv}(G)+2\sum\limits_{C\in \mathfrak {C}_{G}(e)}(-1)^{|V(C)|}{\rm per }L_{V(C)}(G).
\end{eqnarray*}
\end{corollary}
The hook immanant of $Q(G)$ satisfies the following recursive relation:
\begin{corollary}
$(i)$ Let $v$ be a vertex of  $G$. Then
\begin{eqnarray*}
d_k(Q(G))
&=&  d(v)[d_{k-1}(Q_v(G)) + d_k(Q_v(G))] \\
&& + \sum_{u\in N(v)} [d_{k-2}(Q_{uv}(G)) - d_k(Q_{uv}(G))] \\
&& +2 \sum_{C\in \mathfrak{C}_G(v)} [d_{k-|V(C)|}(Q_{V(C)}(G)) -(-1)^{|V(C)|}d_k(Q_{V(C)}(G))].
\end{eqnarray*}

$(ii)$ Let $e=uv$ be an edge of  $G$. Then
\begin{eqnarray*}
d_k(Q(G))
&=& d_k(Q(G-e)) + [d_k(Q_v(G-e)) + d_{k-1}(Q_v(G-e))] \\
&& + [d_k(Q_u(G-e)) + d_{k-1}(Q_u(G-e))] \\
&& +2[d_{k-1}(Q_{uv}(G)) + d_{k-2}(Q_{uv}(G))] \\
&& +2 \sum_{C\in \mathfrak{C}_G(e)} [d_{k-|V(C)|}(Q_{V(C)}(G)) - (-1)^{|V(C)|}d_k(Q_{V(C)}(G))].
\end{eqnarray*}
\end{corollary}

We propose a recursive formula for the hook immanant calculation of $A(G)$.
\begin{corollary}\label{the4.5}
$(i)$ Let $v$ be a vertex of graph $G$. Then
\begin{eqnarray*}
d_k(A(G))
&=& \sum_{u\in N(v)} [d_{k-2}(A(G-u-v)) - d_k(A(G-u-v))] \\
&& +2 \sum_{C\in \mathfrak{C}_G(v)} [d_{k-|V(C)|}(A(G-V(C))) - (-1)^{|V(C)|}d_k(A(G-V(C)))].
\end{eqnarray*}

$(ii)$ Let $e=uv$ be an edge of graph $G$. Then
\begin{eqnarray*}
d_k(A(G))
&=& d_k(A(G-e)) + [d_{k-2}(A(G-u-v))-d_k(A(G-u-v))] \\
&& +2 \sum_{C\in \mathfrak{C}_G(e))} [d_{k-|V(C)|}(A(G-V(C))) - (-1)^{|V(C)|}d_k(A(G-V(C)))].
\end{eqnarray*}
\end{corollary}
Finally,  we give a recurrence relation for the hook immanant of $A_{\alpha}(G)$ as follows.
\begin{corollary}\label{the4.7}
$(i)$ Let $v$ be a vertex of graph $G$. Then
\begin{eqnarray*}
&&d_k(A_{\alpha}(G))\\
&=& \alpha d(v)[d_{k-1}((A_{\alpha})_v(G)) + d_k((A_{\alpha})_v(G))] \\
&& + (1-\alpha)^2 \sum_{u\in N(v)} [d_{k-2}((A_{\alpha})_{uv}(G)) - d_k((A_{\alpha})_{uv}(G))] \\
&& +2 \sum_{C\in \mathfrak{C}_G(v)} (1-\alpha)^{|V(C)|}[d_{k-|V(C)|}((A_{\alpha})_{V(C)}(G)) -(-1)^{|V(C)|}d_k((A_{\alpha})_{V(C)}(G))].
\end{eqnarray*}

$(ii)$ Let $e=uv$ be an edge of graph $G$. Then
\begin{eqnarray*}
&&d_k(A_{\alpha}(G))\\
&=& d_k(A_{\alpha}(G-e)) + \alpha[d_k((A_{\alpha})_v(G-e)) + d_{k-1}((A_{\alpha})_v(G-e))] \\
&& + \alpha[d_k((A_{\alpha})_u(G-e)) + d_{k-1}((A_{\alpha})_u(G-e))] \\
&& +[(2\alpha -1)d_k((A_{\alpha})_{uv}(G)) + 2\alpha^2 d_{k-1}((A_{\alpha})_{uv}(G)) + (2\alpha^2 +2\alpha+1)d_{k-2}((A_{\alpha})_{uv}(G))] \\
&& +2 \sum_{C\in \mathfrak{C}_G(e)} (1-\alpha)^{|V(C)|}[d_{k-|V(C)|}((A_{\alpha})_{V(C)}(G)) - (-1)^{|V(C)|}d_k((A_{\alpha})_{V(C)}(G))].
\end{eqnarray*}
\end{corollary}
\begin{table}[htbp]
\begin{center}
\caption{Several important hook immanantal polynomials $\Phi_k(M,x)$ of graph matrices of order $n$.}
\label{tab1}
\setlength{\tabcolsep}{0.1mm}
\begin{tabular}{|c|c|c|c|c|} 

    \hline  
             &$\beta=1$,$\gamma=-1$             &$\beta=1$,$\gamma=1$              &$\beta=0,\gamma=1$               &$\beta=\alpha,\gamma=1-\alpha$  \\
    \hline
    \makecell{\\ $k=1$\\ }  &  \makecell{$\varphi(L(G))$ \\ (Guo et al., \cite{guo})}   &\makecell{$\varphi(Q(G))$ \\ (Guo et al., \cite{guo})}  & \makecell{$\varphi(A(G))$ \\ (Cvetkovi\'{c} et al., \cite{cvet})}  &$\varphi(A_{\alpha}(G))$  \\
     \cline{2-5} 
     &$\varphi(L(\overrightarrow{G}))$&$\varphi(Q(\overrightarrow{G}))$ &$\varphi(A(\overrightarrow{G}))$&$\varphi(A_{\alpha}(\overrightarrow{G}))$  \\
\hline
     \makecell{\\ $k=2$\\ }   & \makecell{$\Phi_2(L(G))$ \\ (Wu et al., \cite{wu2})}    & $\Phi_2(Q(G))$  &$\Phi_2(A(G))$  &$\Phi_2(A_{\alpha}(G))$  \\
      \cline{2-5} 
             & $\Phi_2(L(\overrightarrow{G}))$  & $\Phi_2(Q(\overrightarrow{G}))$  &$\Phi_2(A(\overrightarrow{G}))$  &$\Phi_2(A_{\alpha}(\overrightarrow{G}))$  \\
     \hline
     \makecell{\\ $k=n$\\ }   & \makecell{$\psi(L(G))$ \\ (Liu and Wu, \cite{liu})}    &\makecell{$\psi(Q(G))$ \\ (Liu and Wu, \cite{liu})} &\makecell{ $\psi(A(G))$ \\ (Borowiecki and J\'{o}zwia\'{k}, \cite{bor})}  &$\psi(A_{\alpha}(G))$ \\
       \cline{2-5} 
             & $\psi(L(\overrightarrow{G}))$  & $\psi(Q(\overrightarrow{G}))$  & $\psi(A(\overrightarrow{G}))$&$\psi(A_{\alpha}(\overrightarrow{G}))$ \\
    \hline  
\end{tabular}
\end{center}
\end{table}

\section{Conclusion}

In this paper, we have derived the recursive formulas for the hook immanantal polynomials and hook immanants of $H(G)$ (resp. $H(\overrightarrow{G})$). 
\begin{table}[htbp]
\begin{center}
\caption{Several important hook immanants $\Phi_k(M)$ of  graph matrices of order $n$. }
\label{tab2}
\setlength{\tabcolsep}{0.25mm}
\begin{tabular}{|c|c|c|c|c|} 

    \hline  
             &$\beta=1$,$\gamma=-1$             &$\beta=1$,$\gamma=1$              &$\beta=0,\gamma=1$               &$\beta=\alpha,\gamma=1-\alpha$  \\
    \hline
\makecell{\\ $k=1$ }  &${\rm det}L(G)$ &${\rm det}Q(G)$ &${\rm det}A(G)$&${\rm det}A_{\alpha}(G)$  \\
     \cline{2-5} 
     &${\rm det}L(\overrightarrow{G})$&${\rm det}Q(\overrightarrow{G})$ &${\rm det}A(\overrightarrow{G})$&${\rm det}A_{\alpha}(\overrightarrow{G})$  \\
\hline
\makecell{\\ $k=2$ }   & $d_2(L(G))$  & $d_2(Q(G))$  &$d_2(A(G))$  &$d_2(A_{\alpha}(G))$  \\
      \cline{2-5} 
             & $d_2(L(\overrightarrow{G}))$  & $d_2(Q(\overrightarrow{G}))$  &$d_2(A(\overrightarrow{G}))$  &$d_2(A_{\alpha}(\overrightarrow{G}))$  \\
     \hline
\makecell{\\ $k=n$ }   &  \makecell{${\rm per}L(G)$ \\ (Wu et al., \cite{wu1})}   & ${\rm per}Q(G)$  & ${\rm per}A(G)$&${\rm per}A_{\alpha}(G)$ \\
       \cline{2-5} 
             & ${\rm per}L(\overrightarrow{G})$  & ${\rm per}Q(\overrightarrow{G})$  & ${\rm per}A(\overrightarrow{G})$&${\rm per}A_{\alpha}(\overrightarrow{G})$ \\
    \hline  
\end{tabular}
\end{center}
\end{table}
By the results of Theorems \ref{thm1.1} and \ref{thm1.3}, we get recursive formulas for the hook immanantal polynomials of the graph matrices
$L(G)$, $Q(G)$, $A(G)$, and $A_{\alpha}(G)$ (resp. $L(\overrightarrow{G})$, $Q(\overrightarrow{G})$, $A(\overrightarrow{G})$ and $A_{\alpha}(\overrightarrow{G})$). Similarly,  we aslo can obtain recursive formulas for the hook immanants of these matrices. Moreover, we are able to deduce the existing results on the characteristic polynomial, the second immanantal polynomial and the permanental polynomial (resp. determinant, second immanant and permanent) in algebraic graph theory. As for other hook immanantal polynomials that remain unstudied, they can also be derived. For details,  see Tables \ref{tab1} and \ref{tab2}. Furthermore, we can obtain recursive formulas on  $\Phi_k(M,x)$ and $\Phi_k(M)$, where $k = 3, 4,\ldots, n-1$ and $M$ are $L(G)$, $Q(G)$, $A(G)$, and $A_{\alpha}(G)$.

\noindent{\bf Data Availability}\\
{No data were used to support this study.}

\noindent{\bf Acknowledgements:} This research is supported by NSFC (Nos. 12261071, 12471333)  and NSF of Qinghai Province (No. 2025-ZJ-902T).

\end{document}